\documentclass[10pt]{amsart}
\usepackage{latexsym}
\usepackage{amsmath}
\usepackage{amssymb}
\usepackage{mathrsfs}
\usepackage{graphicx}
\usepackage{color}
\usepackage{pgfpages}
\usepackage{ifthen}
\usepackage{leftidx,tensor}
\usepackage[T1]{fontenc}
\usepackage[latin1]{inputenc}
\usepackage{mathtools}
\usepackage{comment}
\usepackage{dsfont}
\usepackage[nocompress]{cite}

\usepackage[shortlabels]{enumitem}
\usepackage{aliascnt}
\usepackage[bookmarks=true,pdfstartview=FitH, pdfborder={0 0 0}, colorlinks=true,citecolor=red, linkcolor=blue]{hyperref}
\usepackage{bbm}
\usepackage{nicefrac}

\theoremstyle{plain}
\newtheorem{thm}{Theorem}[section]
\newaliascnt{cor}{thm}
\newaliascnt{prop}{thm}
\newaliascnt{lem}{thm}
\newtheorem{cor}[cor]{Corollary}
\newtheorem{coron}[cor]{Corollary}
\newtheorem{prop}[prop]{Proposition}
\newtheorem{lem}[lem]{Lemma}
\aliascntresetthe{cor}
\aliascntresetthe{prop}
\aliascntresetthe{lem} 
\theoremstyle{definition}
\newaliascnt{defn}{thm}
\newaliascnt{asu}{thm}
\newaliascnt{con}{thm}
\newtheorem{defn}[defn]{Definition}
\newtheorem{asu}[asu]{Assumption}

\aliascntresetthe{defn}
\aliascntresetthe{asu}
\aliascntresetthe{con}
\newcounter{stp}
\newcounter{stpi}
\newcounter{stpci}
\newcounter{stpiii}

\theoremstyle{remark}
\newaliascnt{rem}{thm}
\newaliascnt{exa}{thm}
\newaliascnt{masu}{thm}
\newaliascnt{nota}{thm}
\newaliascnt{sett}{thm}
\newtheorem{rem}[rem]{Remark}

\aliascntresetthe{rem}
\aliascntresetthe{exa}
\aliascntresetthe{masu}
\aliascntresetthe{nota}
\aliascntresetthe{sett}
\numberwithin{equation}{section}

\setlist[enumerate]{font = \normalfont}

\newcommand{\eps}{\varepsilon}

\newcommand{\atm}{\mathrm{atm}}
\newcommand{\ocn}{\mathrm{ocn}}
\newcommand{\ice}{\mathrm{ice}}

\renewcommand{\cor}{\mathrm{cor}}

\newcommand{\Cmr}{C_\mathrm{MR}}
\newcommand{\bR}{\mathbb{R}}

\newcommand{\bN}{\mathbb{N}}

\newcommand{\bE}{\mathbb{E}}

\newcommand{\bS}{\mathbb{S}}

\newcommand{\bA}{\mathbb{A}}
\newcommand{\obB}{\overline{\mathbb{B}}}
\newcommand{\rd}{\mathrm{d}}
\newcommand{\rb}{\mathrm{b}}
\newcommand{\mre}{\mathrm{e}}
\newcommand{\mri}{\mathrm{i}}
\newcommand{\srd}{\, \mathrm{d}}

\newcommand{\tu}{\Tilde{u}}
\newcommand{\tv}{\Tilde{v}}

\newcommand{\teps}{\Tilde{\eps}}
\newcommand{\tB}{\Tilde{B}}
\newcommand{\tBh}{\Tilde{B}_{\mathrm{h}}}
\newcommand{\tBa}{\Tilde{B}_{\mathrm{a}}}

\newcommand{\rC}{\mathrm{C}}
\newcommand{\rL}{\mathrm{L}}
\newcommand{\rW}{\mathrm{W}}
\newcommand{\rH}{\mathrm{H}}
\newcommand{\rB}{\mathrm{B}}
\newcommand{\rD}{\mathrm{D}}

\newcommand{\rX}{\mathrm{X}}
\newcommand{\rY}{\mathrm{Y}}
\newcommand{\rZ}{\mathrm{Z}}

\newcommand{\rBUC}{\mathrm{BUC}}

\newcommand{\rI}{\mathrm{I}}
\newcommand{\rII}{\mathrm{II}}
\newcommand{\rIII}{\mathrm{III}}
\newcommand{\rIV}{\mathrm{IV}}

\newcommand{\rLp}{\rL^p}

\newcommand{\cA}{\mathcal{A}}

\newcommand{\cL}{\mathcal{L}}

\newcommand{\cK}{\mathcal{K}}

\newcommand{\vice}{v_{\ice}}
\newcommand{\vii}{v_{\ice,0}}

\newcommand{\tatm}{\tau_{\atm}}
\newcommand{\tocn}{\tau_{\ocn}}
\newcommand{\tice}{\tau_{\ice}}

\DeclareMathOperator{\Id}{Id}

\DeclareMathOperator{\diag}{diag}
\DeclareMathOperator{\tr}{tr}

\DeclareMathOperator{\dist}{dist}
\DeclareMathOperator{\Cof}{Cof}
\newcommand{\Hinfty}{\mathcal{H}^\infty}

\newcommand{\tri}{\triangle}
\newcommand{\trid}{\triangle_\delta}
\newcommand{\zetad}{\zeta_\delta}
\newcommand{\etad}{\eta_\delta}
\newcommand{\sigmad}{\sigma_\delta}

\newcommand{\tvice}{\Tilde{v}_\ice}
\renewcommand{\th}{\Tilde{h}}
\newcommand{\ta}{\Tilde{a}}
\newcommand{\hvice}{\widehat{v}_\ice}
\newcommand{\hh}{\widehat{h}}
\newcommand{\ha}{\widehat{a}}
\newcommand{\hu}{\widehat{u}}
\newcommand{\hv}{\widehat{v}}

\newcommand{\Sh}{S_{\mathrm{h}}}
\newcommand{\Sa}{S_{\mathrm{a}}}
\newcommand{\fgr}{f_{\mathrm{gr}}}

\newcommand{\dOmega}{\del \Omega}
\newcommand{\TOmega}{(0,T) \times \Omega}
\newcommand{\TdOmega}{(0,T) \times \dOmega}

\newcommand{\oOmega}{\overline{\Omega}}

\newcommand{\ccor}{c_{\cor}}

\newcommand{\Catm}{C_{\atm}}
\newcommand{\Cocn}{C_{\ocn}}

\newcommand{\Ratm}{R_{\atm}}
\newcommand{\Rocn}{R_{\ocn}}
\newcommand{\ratm}{\rho_{\atm}}
\newcommand{\rocn}{\rho_{\ocn}}
\newcommand{\rice}{\rho_{\ice}}
\newcommand{\mice}{m_{\ice}}

\newcommand{\xH}{x_{\rH}}
\newcommand{\yH}{y_{\rH}}

\newcommand{\divH}{\mathrm{div}_{\rH} \,}
\newcommand{\nablaH}{\nabla_{\rH}}

\newcommand{\AHD}{A_{\rD}^{\rH}}

\newcommand{\bAH}{\bA^{\rH}}
\newcommand{\tbAH}{\Tilde{\bA}^{\rH}}

\newcommand{\Vatm}{V_\atm}
\newcommand{\Vocn}{V_\ocn}

\newcommand{\del}{\partial}
\newcommand{\dk}[1]{\partial_{#1}}
\newcommand{\dt}{\dk{t}} 
 
\newcommand{\tin}{\enspace \text{in} \enspace}
\newcommand{\ton}{\enspace \text{on} \enspace}
\newcommand{\tfor}{\enspace \text{for} \enspace}
\newcommand{\tforall}{\enspace \text{for all} \enspace}
\newcommand{\tand}{\enspace \text{and} \enspace}

\newcommand{\twith}{\enspace \text{with} \enspace}
\newcommand{\tso}{\enspace \text{so} \enspace}

\newcommand{\taswellas}{\enspace \text{as well as} \enspace}

\begin{document}

\title[Strong well-posedness of the Hibler sea ice model]{Well-posedness of Hibler's parabolic-hyperbolic sea ice model}

\author{Felix Brandt}
\address{Department of Mathematics, University of California at Berkeley, Berkeley, 94720, CA, USA.}
\email{fbrandt@berkeley.edu}
\subjclass[2020]{35Q86, 35K59, 86A05, 86A10}
\keywords{Hibler's sea ice model, local strong well-posedness, viscous-plastic material, parabolic-hyperbolic problem, quasilinear evolution equation, Lagrangian coordinates, anisotropic ground space}

\begin{abstract}
This paper proves the local-in-time strong well-posedness of a parabolic-hyperbolic regularized version of Hibler's sea ice model.
Hibler's model is the most frequently used sea ice model in climate science.
Lagrangian coordinates are employed to handle the hyperbolic terms in the balance laws.
The resulting problem is regarded as a quasilinear non-autonomous evolution equation.
Maximal $\rL^p$-regularity of the underlying linearized problem is obtained on an anisotropic ground space in order to deal with the lack of regularization in the balance laws.
\end{abstract}

\maketitle

\section{Introduction}

Recent years have witnessed an increase in research on sea ice dynamics both, from the theoretical and the applied side.
One reason for this might be the role of sea ice in climate science, acting as an insulator in between the atmosphere and the ocean.
For more details on the importance of sea ice in climate studies, we refer e.\ g.\ to the survey article of Hunke, Lipscomb and Turner \cite{HLT:10}. 
Golden et al.\ \cite{Gol:20} provide an overview of sea ice modeling.
To date, almost all climate models capturing sea ice dynamics rely on Hibler's viscous-plastic model introduced in \cite{Hib:79} in 1979.
It is a 2D large-scale model that describes dynamic as well as thermodynamic aspects of sea ice.
In particular, the complex mechanical behavior of sea ice is taken into account in the viscous-plastic rheology.

First well-posedness results to Hibler's model were obtained by Brandt, Disser, Haller-Dintelmann and Hieber \cite{BDHH:22} as well as Liu, Thomas and Titi \cite{LTT:22}.
In \cite{BDHH:22}, a fully parabolic variant of Hibler's model is investigated, and local-in-time strong well-posedness and global-in-time well-posedness close to equilibrium solutions are shown.
The parabolic-hyperbolic nature of the model is stressed in \cite{LTT:22}, but the authors strongly modify the stress tensor and establish local strong well-posedness of the resulting problem.

It is the aim of the present paper to prove the local strong well-posedness of the {\em parabolic-hyperbolic} variant of Hibler's model, where the {\em most commonly used regularization for the stress tensor} is employed.
Unlike the models considered in \cite{BDHH:22} or \cite{LTT:22}, this is the model typically studied in numerical simulations.

In contrast to the rigorous mathematical analysis, which only started quite recently, there is vast literature on numerical analysis and simulations of Hibler's model, see for example the work of Zhang and Hibler~\cite{ZH:91}.
In order to ease computational costs, Hunke and Dukowicz \cite{HD:97} suggested an elastic-viscous-plastic rheology.
This model was later on modified in order to reduce the discrepancy with observations, see e.\ g.\ the article of Kimmrich, Danilov and Losch \cite{KDL:15}.
Lemieux and Tremblay \cite{LT:09} implemented an inexact Newton method for Hibler's model.
In \cite{MR:17}, Mehlmann and Richter developed a modified Newton solver. 
For further works in this direction, we also refer to the articles of Seinen and Khouider \cite{SK:18}, Shih, Mehlmann, Losch and Stadler \cite{SMLS:23} or Yaremchuk and Panteleev \cite{YP:22}.
Mehlmann and Korn \cite{MK:21} provided numerical studies of sea ice on triangular grids for an integration into the ocean general circulation model ICON-O.  

As indicated above, we invoke a regularized stress tensor here.
In fact, although the original stress tensor describes an idealized viscous-plastic material, it gives rise to a degenerate problem, see also \autoref{sec:main result}.
As a result, already Hibler \cite{Hib:79} suggested to cut off the viscosities.
In the present paper, we follow the most common regularization of the stress tensor as e.\ g.\ suggested by Kreyscher et al.\ \cite{KHLFG:00}.
Recently, different possibilities to modify the rheology were suggested by Chatta, Khouider and Kesri \cite{CKK:23} and Ringeisen, Losch and Tremblay \cite{RLT:23}.

For the analysis, we employ the Lagrangian change of coordinates in order to deal with the hyperbolic effects in the balance laws. 
The advantage is that the hyperbolic terms in the balance laws are annihilated by the introduction of Lagrangian coordinates.
More precisely, the time derivative with respect to Lagrangian coordinates corresponds precisely to the material derivative in the Eulerian formulation.
For the Lagrangian approach in the context of the compressible Navier-Stokes equations, we refer for instance to the article of Danchin \cite{Dan:14}.
Let us stress, however, that the present problem differs significantly from compressible Navier-Stokes equations.
First, the regularized stress tensor $\sigmad$ shown in \eqref{eq:reg stress tensor sea ice} depends on the deformation tensor $\eps = \nicefrac{1}{2}(\nablaH \vice + (\nablaH \vice)^\top)$ in a nonlinear way.
In addition, the pressure is replaced by a so-called ice strength $P = P(h,a)$ as displayed in \eqref{eq:ice strength} which is in turn a function of the ice thickness~$h$ and the ice concentration $a$ for which balance laws are prescribed.

For the principal variable $u = (\vice,h,a)$ capturing the horizontal ice velocity $\vice$, the ice thickness~$h$ and the ice concentration $a$, a suitable operator $A(t,u)$ and right-hand side $F(t,u)$, we rewrite the resulting {\em non-autonomous} and {\em quasilinear} system in Lagrangian coordinates as a Cauchy problem of the form 
\begin{equation*}
    \frac{\rd}{\rd t} u + A(t,u)u = F(t,u), \tfor t > 0, \enspace u(0) = u_0,
\end{equation*}
on a Banach space $\rX_0$.
For the linearized problem, we employ the so-called {\em maximal regularity} method.
Roughly speaking, this means that for $A(u_0) \coloneqq A(0,u_0) \colon \rX_1 \to \rX_0$ and the time trace $\tr$, the operator
\begin{equation*}
    \Bigl(\frac{\rd}{\rd t} + A(u_0),\tr\Bigr) \colon \bE_1 \to \bE_0 \times \rX_\gamma
\end{equation*}
is an isomorphism between suitable Banach spaces $\bE_1$ and $\bE_0 \times \rX_\gamma$.
In this paper, we choose the $\rL^p$-framework, so we consider solutions in $\bE_1 \coloneqq \rW^{1,p}(0,T;\rX_0) \cap \rL^p(0,T;\rX_1)$, meaning that every term in the differential equation lies in $\bE_0 \coloneqq \rL^p(0,T;\rX_0)$.
The space $\rX_\gamma$ then is the real interpolation space $(\rX_0,\rX_1)_{1-\nicefrac{1}{p},p}$.
For further details on maximal regularity in general and maximal $\rL^p$-regularity in particular, we refer for instance to the monographs of Amann \cite{Ama:95}, Lunardi \cite{Lun:95} or Pr\"uss and Simonett \cite{PS:16}, the memoir of Denk, Hieber and Pr\"uss \cite{DHP:03} or the book chapter of Kunstmann and Weis \cite{KW:04}. 
It is a key step in our analysis to introduce an anisotropic ground space of the form $\rX_0 = \rL^q \times \rW^{1,q} \times \rW^{1,q}$ for the ice velocity, ice thickness and ice concentration to obtain the maximal $\rL^p$-regularity of the linearized problem.
In fact, we even show that the associated operator matrix admits a bounded $\Hinfty$-calculus by means of theory for diagonally dominant block operator matrices and deduce the maximal $\rL^p$-regularity from there.
We then use the maximal $\rL^p$-regularity to set up a fixed point argument in the Lagrangian formulation.
For this, a precise tracking of the change of coordinates is required to derive suitable nonlinear estimates.

Recently, sea ice has also been studied in different contexts.
In \cite{BH:23}, sea ice is investigated in the time periodic setting, see also \cite[Section~7.2]{Bra:24}.
Moreover, local strong well-posedness of the interaction problem of sea ice with a rigid body is established in \cite{BBH:24a}.
The paper \cite{BBH:24b} introduces sea ice as a thin layer in between the atmosphere and the ocean and establishes local strong well-posedness as well as global strong well-posedness close to equilibria of the coupled model.
The motivation is the coupled atmosphere-ocean model introduced by Lions, Temam and Wang \cite{LTW:93, LTW:95}.
In the papers \cite{BH:23, BBH:24a, BBH:24b}, the fully parabolic regularized variant of Hibler's model as analyzed in \cite{BDHH:22} is taken into account.
Let us also mention the recent work of Piersanti and Temam \cite{PT:23} on the modeling and analysis of the dynamics of shallow ice sheets.

For convenience of the reader, we briefly elaborate on the function spaces used in this paper.
Consider a bounded domain $\Omega \subset \bR^2$ with boundary of class $\rC^2$, and suppose that $G$ either represents $\Omega$, or an interval of the form $(0,T)$, for $T > 0$.
With $p \in [1,\infty]$ and $k \in \bN$, we denote by $\rL^p(G)$ and $\rW^{k,p}(G)$ the Lebesgue and Sobolev spaces on $G$, respectively.
The space $\rW_0^{1,q}(\Omega)$, $q \in (1,\infty)$, contains the functions in~$\rW^{1,q}(\Omega)$ whose trace vanishes on $\del \Omega$.
Moreover, for $l \in \{0,1\}$ and $\alpha \in (0,1)$, the space~$\rC^{l,\alpha}(\oOmega)$ consists of all functions $f \in \rC^l(\oOmega)$ such that all partial derivatives $\del^\beta f$ of orders $|\beta| \le l$ are $\alpha$-H\"older continuous on $\oOmega$.
By $\rC_\rb^1([0,\infty))$, we denote the functions in $\rC^1([0,\infty))$ that are bounded on $[0,\infty)$, and whose derivative is additionally bounded on $[0,\infty)$.
The space $\rBUC([0,T])$ represents the bounded and uniformly continuous functions on $[0,T]$.
For $s > 0$ and $p$, $q \in (1,\infty)$, we use $\rB_{qp}^s(\Omega)$ for the Besov spaces on $\Omega$.
Note that they can be obtained from the Lebesgue and Sobolev spaces by real interpolation, see e.\ g.\ \cite[Sec.~2.4 and 4.3.1]{Tri:78}.
More precisely, for $s$, $p$, $q$ as above, we have $\rB_{qp}^s(\Omega) = (\rL^q(\Omega),\rW^{2,q}(\Omega))_{\frac{s}{2},p}$.
Finally, for a Banach space $\rX$, we denote by $\rL^p(0,T;\rX)$ and $\rW^{1,p}(0,T;\rX)$ the respective Bochner spaces.

This article is organized as follows.
In \autoref{sec:main result}, we introduce Hibler's model and state the main result, \autoref{thm:local strong wp sea ice para-hyper}, on the local-in-time strong well-posedness.
\autoref{sec:trafo Lagrangian coords} is then dedicated to the transformation to Lagrangian coordinates and to the reformulation of the main result, see \autoref{thm:local strong wp Lagrangian coords}.
In addition, we discuss the continuous dependence on the initial data as well as a blow-up criterion in the Lagrangian formulation.
The focal point of \autoref{sec:lin theory} is the proof of the bounded $\Hinfty$-calculus and the maximal $\rL^p$-regularity of the linearized operator matrix with the anisotropic ground space.
In \autoref{sec:proof main result}, we prove the main result.
For this, we start with the reformulation as a fixed point problem.
Thereafter, we derive the estimates of the nonlinear terms.
We then deduce \autoref{thm:local strong wp sea ice para-hyper} from \autoref{thm:local strong wp Lagrangian coords} by the invertibility of the change of coordinates for small time.
\autoref{sec:outlook and open problems} presents some concluding remarks.

\section{Hibler's sea ice model and main result}\label{sec:main result}

In this section, we first recall Hibler's viscous-plastic sea ice model from \cite{Hib:79} and then state the main result on the local strong well-posedness.
In the sequel, we consider a bounded domain $\Omega \subset \bR^2$ with boundary $\dOmega$ of class $\rC^2$ and a time interval $(0,T)$, where $0 < T < \infty$.
The model variables are the horizontal sea ice velocity $\vice \colon \TOmega \to \bR^2$, the mean ice thickness $h \colon \TOmega \to [\kappa,\infty)$, with $\kappa > 0$ denoting a small parameter, and the ice concentration $a \colon \TOmega \to (0,1)$.
The latter variable describes the ratio of area covered by thick ice, i.\ e., sea ice with a thickness larger than $h_\bullet$, where $h_\bullet > 0$ is a certain thickness level.
The lower bound of the mean ice thickness means that in each control area, there is at least some sea ice.
In the sequel, we denote by $u$ the principle variable, i.\ e., $u = (\vice,h,a)$.

The characteristics of sea ice as a viscous-plastic material are encoded in the stress tensor introduced below.
In the following, $\eps = \eps(\vice) = \nicefrac{1}{2}(\nablaH \vice + (\nablaH \vice)^\top)$ represents the deformation tensor.
We remark that $\nablaH$ is the gradient in two variables.
The same is valid for $\divH$, denoting the horizontal divergence.
Moreover, for given constants $p^* > 0$ and $c_\bullet > 0$, the ice strength $P$ takes the explicit shape
\begin{equation}\label{eq:ice strength}
    P = P(h,a) = p^* h \mre^{-c_\bullet (1-a)}.
\end{equation}
It can be found in \cite[p.~822, eq.~(17)]{Hib:79}.
For the ratio $e > 1$ of the major to minor axes of the elliptical yield curve of the principle components of the stress, we introduce
\begin{equation*}
    \tri^2(\eps) \coloneqq \left(\eps_{11}^2 + \eps_{22}^2\right)\Bigl(1+\frac{1}{e^2}\Bigr) + \frac{4}{e^2} \eps_{12}^2 + 2 \eps_{11} \eps_{22} \Bigl(1-\frac{1}{e^2}\Bigr).
\end{equation*}
The bulk and shear viscosities are as made precise in \cite[p.~819, eq.~(7) and~(8)]{Hib:79}, and they take the shape
\begin{equation*}
    \zeta(\eps,P) = \frac{P(h,a)}{2 \tri(\eps)} \tand \eta(\eps,P) = e^{-2} \zeta(\eps,P).
\end{equation*}
The stress tensor $\sigma$ accounting for the behavior of sea ice as a viscous-plastic material then reads as
\begin{equation*}
    \sigma = 2 \eta(\eps,P) \eps + [\zeta(\eps,P) - \eta(\eps,P)] \tr(\eps) \Id_2 - \frac{P}{2} \Id_2,
\end{equation*}
with $\Id_2$ denoting the identity matrix.
In spite of the fact that the above law describes an idealized viscous-plastic behavior, the bulk and shear viscosities $\zeta$ and $\eta$ become singular for $\eps \to 0$.
For this reason, we invoke the same regularization as in \cite{MK:21} and \cite{BDHH:22}, see also \cite{KHLFG:00}, by defining
\begin{equation*}
    \trid(\eps) \coloneqq \sqrt{\delta + \tri^2(\eps)}
\end{equation*}
for $\delta > 0$.
Accordingly, we introduce the regularized viscosities
\begin{equation*}
    \zetad(\eps,P) \coloneqq \frac{P(h,a)}{2 \trid(\eps)} \tand \etad(\eps,P) \coloneqq e^{-2} \zetad(\eps,P),
\end{equation*}
resulting in the regularized stress tensor $\sigmad$ studied in the remainder of this article and given by
\begin{equation}\label{eq:reg stress tensor sea ice}
    \sigmad \coloneqq 2 \etad(\eps,P) \eps + [\zetad(\eps,P) - \etad(\eps,P)] \tr(\eps) \Id_2 - \frac{P}{2} \Id_2.
\end{equation}

Hibler's model consists of a momentum equation for the sea ice velocity $\vice$ coupled to balance laws for~$h$ and $a$.
The momentum balance includes the internal ice stress expressed by $\divH \sigmad$ as well as external forcing terms and takes the shape
\begin{equation*}
    \mice \left(\dt \vice + (\vice \cdot \nablaH) \vice\right) = \divH \sigmad - \mice \ccor \vice^\perp - \mice g \nablaH H + \tatm + \tocn(\vice).
\end{equation*}
In the above $\mice = \rice h$ represents the ice mass for a constant ice density $\rice > 0$, $\ccor > 0$ is the Coriolis parameter, $g$ denotes the gravity, $H \colon \TOmega \to [0,\infty)$ is the sea surface dynamic height and~$\tatm$ as well as $\tocn(\vice)$ correspond to the atmospheric wind and oceanic forces.
The latter two read as
\begin{equation*}
    \tatm = \ratm \Catm |\Vatm| \Ratm \Vatm \tand \tocn(\vice) = \rocn \Cocn |\Vocn - \vice| \Rocn (\Vocn - \vice)
\end{equation*}
for constant densities $\ratm > 0$ and $\rocn > 0$, constant drag coefficients $\Catm > 0$ and $\Cocn > 0$, velocities~$\Vatm$ and $\Vocn$ of the atmosphere and the ocean and rotation matrices $\Ratm$ and $\Rocn$.

The balance laws of the mean ice thickness $h$ and the ice concentration $a$ are given by
\begin{equation*}
\left\{
    \begin{aligned}
        \dt h + \divH(\vice h)
        &= \Sh(h,a),\\
        \dt a + \divH(\vice a)
        &= \Sa(h,a),
    \end{aligned}
\right.
\end{equation*}
where $\Sh$ and $\Sa$ represent thermodynamic source terms.
For a function $\fgr \in \rC_{\rb}^1([0,\infty))$ modeling the ice growth rate, see for instance the one suggested by Hibler \cite{Hib:79}, these terms are
\begin{equation*}
    \begin{aligned}
        \Sh(h,a) 
        &= \fgr\left(\frac{h}{a}\right)a + (1-a)\fgr(0) \tand\\
        \Sa(h,a)
        &= \begin{cases*} \frac{\fgr(0)}{\kappa}(1-a),& if $\fgr(0) > 0$, \\ 0, & if $\fgr(0) < 0$,
        \end{cases*}
         \enspace + \enspace \begin{cases*} 0,& if $\Sh > 0$, \\ \frac{a}{2 h}\Sh, & if $\Sh < 0$.
    \end{cases*} 
    \end{aligned}
\end{equation*}

In the present setting, the model is completed by Dirichlet boundary conditions for the sea ice velocity, so $\vice = 0$ on $\TdOmega$, while no boundary conditions for $h$ and $a$ are assumed.
Moreover, the initial conditions read as $\vice(0) = \vii$, $h(0) = h_0$ and $a(0) = a_0$.

The previous considerations are summarized in the complete system of equations
\begin{equation}\label{eq:parabolic-hyperbolic regularized model}
    \left\{
    \begin{aligned}
        \dt \vice + (\vice \cdot \nablaH) \vice 
        &= \frac{1}{\mice}\divH \sigmad - \ccor \vice^\perp - g \nablaH H\\
        &\quad + \frac{1}{\mice}(\tatm + \tocn(\vice)), &&\tin \TOmega,\\
        \dt h + \divH(\vice h)
        &= \Sh(h,a), &&\tin \TOmega,\\
        \dt a + \divH(\vice a)
        &= \Sa(h,a), &&\tin \TOmega,\\
        \vice &
        = 0, &&\ton \TdOmega,\\
        \vice(0) = \vii, \enspace h(0) 
        &= h_0, \enspace a(0) = a_0, &&\tin \Omega.
    \end{aligned}
    \right.
\end{equation}

Before stating the main result, we require some further preparation with regard to notation which we address below.
For brevity, we define the spaces $\rX_0$ and $\rX_1$ by
\begin{equation}\label{eq:ground space and regularity space sea ice para-hyper}
    \rX_0 \coloneqq \rL^q(\Omega)^2 \times \rW^{1,q}(\Omega) \times \rW^{1,q}(\Omega) \tand \rX_1 \coloneqq \rW^{2,q}(\Omega)^2 \cap \rW_0^{1,q}(\Omega)^2 \times \rW^{1,q}(\Omega) \times \rW^{1,q}(\Omega),
\end{equation}
where $q \in (1,\infty)$.
In the sequel, we will assume that $p$, $q \in (1,\infty)$ satisfy
\begin{equation}\label{eq:cond p and q sea ice para-hyper}
    \frac{1}{p} + \frac{1}{q} < \frac{1}{2}.
\end{equation}
With regard to the initial data, for $p$, $q \in (1,\infty)$ fulfilling \eqref{eq:cond p and q sea ice para-hyper}, we then introduce the time trace space~$\rX_\gamma = (\rX_0,\rX_1)_{1-\nicefrac{1}{p},p}$ taking the shape
\begin{equation}\label{eq:trace space sea ice para-hyper}
    \rX_\gamma = \rB_{qp,\rD}^{2-\nicefrac{2}{p}}(\Omega)^2 \times \rW^{1,q}(\Omega) \times \rW^{1,q}(\Omega).
\end{equation}
In the above, the subscript $_\rD$ indicates Dirichlet boundary conditions.
For $p$, $q \in (1,\infty)$ such that \eqref{eq:cond p and q sea ice para-hyper} holds true, it follows from Sobolev embeddings, see for example \cite[Theorem~4.6.1]{Tri:78}, that
\begin{equation}\label{eq:emd trace space sea ice para-hyper}
    \rX_\gamma \hookrightarrow \rB_{qp}^{2 - \nicefrac{2}{p}}(\Omega)^2 \times \rW^{1,q}(\Omega) \times \rW^{1,q}(\Omega) \hookrightarrow \rC^{1,\alpha}(\oOmega)^2 \times \rC^{0,\alpha}(\oOmega) \times \rC^{0,\alpha}(\oOmega),
\end{equation}
for some $\alpha > 0$.
In order to guarantee that the initial values lie in the physically relevant ranges of the respective variables, we introduce an open set $V \subset \rX_\gamma$ with
\begin{equation}\label{eq:open set V sea ice para-hyper}
    V \coloneqq \{u = (\vice,h,a) \in \rX_\gamma : h > \kappa \tand a \in (0,1)\}.
\end{equation}
In view of \eqref{eq:emd trace space sea ice para-hyper}, it is legitimate to make pointwise assumptions on $h$ and $a$ for $u \in \rX_\gamma$ if \eqref{eq:cond p and q sea ice para-hyper} is valid.

For the statement of the main result, we explicitly state the assumptions on the external forcing terms.

\begin{asu}\label{ass:external forcing terms}
Let $q \in (1,\infty)$.
We make the following assumptions on the data.
\begin{enumerate}[(a)]
    \item The wind and ocean velocities $\Vatm$ and $\Vocn$ satisfy $\Vatm$, $\Vocn \in \rL^\infty(0,t;\rL^{2q}(\Omega)^2)$ for all $t > 0$.
    \item For the sea surface dynamic height $H$, it is valid that $\nablaH H \in \rL^\infty(0,t;\rL^q(\Omega)^2)$ for all $t > 0$.
    \item It holds that $\fgr \in \rC_{\rb}^1([0,\infty))$ for the ice growth rate $\fgr$.
\end{enumerate}
\end{asu}

The above assumptions on the integrability and regularity of the data are relatively mild, and it seems that they are physically reasonable.
For examples of ice growth rate functions $\fgr$, which generally depend on the season, we refer for instance to the ones considered by Hibler \cite[p.~821]{Hib:79}.

We are now in the position to formulate the main result of this article.

\begin{thm}\label{thm:local strong wp sea ice para-hyper}
Let $p, q \in (1,\infty)$ be such that \eqref{eq:cond p and q sea ice para-hyper} is valid, let $u_0 \in V$, where $V$ was introduced in \eqref{eq:open set V sea ice para-hyper}, and recall the spaces $\rX_0$ and $\rX_1$ from \eqref{eq:ground space and regularity space sea ice para-hyper}.
Moreover, suppose that the external terms $\Vatm$, $\Vocn$, $H$ and $\fgr$ fulfill \autoref{ass:external forcing terms}.
Then there is $T > 0$ such that \eqref{eq:parabolic-hyperbolic regularized model} has a unique solution $u = (\vice,h,a)$ with
\begin{equation*}
    u \in \rW^{1,p}(0,T;\rX_0) \cap \rL^p(0,T;\rX_1) \cap \rC([0,T];V).
\end{equation*}
\end{thm}

\section{Transformation to Lagrangian coordinates}
\label{sec:trafo Lagrangian coords}

This section is dedicated to the transformation to Lagrangian coordinates which is a major part of the strategy to show the local strong well-posedness.
The coordinate transform from Eulerian to Lagrangian coordinates involves the characteristics associated to the ice velocity $\vice$.
In other words, $X$ solves
\begin{equation}\label{eq:char Euler Lagrange}
    \left\{
    \begin{aligned}
        \dt X(t,\yH)
        &= \vice(t,X(t,\yH)), &&\tfor t > 0,\\
        X(0,\yH)
        &= \yH, &&\tfor \yH \in \bR^2.
    \end{aligned}
    \right.
\end{equation}
Given $t \in (0,T)$, we use $Y(t,\cdot) = [X(t,\cdot)]^{-1}$ to denote the inverse of $X(t,\cdot)$, and we will elaborate on the invertibility in more details later on.
The variables in Lagrangian coordinates are then defined by
\begin{equation}\label{eq:change of var}
    \begin{aligned}
        \tvice(t,\yH)
        &\coloneqq \vice(t,X(t,\yH)), \enspace \th(t,\yH) \coloneqq h(t,X(t,\yH)) \tand \ta(t,\yH) \coloneqq a(t,X(t,\yH)).
    \end{aligned}
\end{equation}
Integrating in time in \eqref{eq:char Euler Lagrange}, and inserting the change of variables from \eqref{eq:change of var}, we deduce that
\begin{equation*}
    X(t,\yH) = \yH + \int_0^t \tvice(s,\yH) \srd s.
\end{equation*}

The main reason for the transform of the Lagrangian coordinates is that this allows us to circumvent the hyperbolic effects in the balance laws.
In fact, it readily follows that the time derivative in Lagrangian coordinates coincides with the material derivative in Eulerian coordinates.
Indeed, we have
\begin{equation*}
    \begin{aligned}
        \dt \tvice(t,\yH)
        &= \dt \vice(t,X(t,\yH)) + (\vice(t,X(t,\yH)) \cdot \nablaH) \vice(t,X(t,\yH)),\\
        \dt \th(t,\yH)
        &= \dt h(t,X(t,\yH)) + \vice(t,X(t,\yH)) \cdot \nablaH h(t,X(t,\yH)) \tand\\
        \dt \ta(t,\yH)
        &= \dt a(t,X(t,\yH)) + \vice(t,X(t,\yH)) \cdot \nablaH a(t,X(t,\yH)).
    \end{aligned}
\end{equation*}

We next compute the other transformed terms from \eqref{eq:parabolic-hyperbolic regularized model}.
For the internal ice stress, we first recall the so-called {\em Hibler operator} related to $\divH \sigmad$, and we define the matrix $\bS \colon \bR^{2 \times 2} \to \bR^{2 \times 2}$ such that
\begin{equation*}
    \bS \eps = \begin{pmatrix}
    \left(1 + \frac{1}{e^2}\right) \eps_{11} + \left(1 - \frac{1}{e^2}\right) \eps_{22} & \frac{1}{e^2} \left(\eps_{12} + \eps_{21}\right) \\
    \frac{1}{e^2} \left(\eps_{12} + \eps_{21}\right) & \left(1 - \frac{1}{e^2}\right) \eps_{11} + \left(1 + \frac{1}{e^2}\right) \eps_{22}
    \end{pmatrix}.
\end{equation*}
An identification of $\eps \in \bR^{2 \times 2}$ with the vector $(\eps_{11},\eps_{12},\eps_{21},\eps_{22})^\top \in \bR^4$ then allows us to interpret the action of $\bS$ to $\eps$ as the multiplication by the matrix
\begin{equation*}
    \bS = \left(\bS_{ij}^{kl}\right) = \begin{pmatrix}
	    1+\frac{1}{e^2} & 0 & 0 & 1-\frac{1}{e^2} \\ 
		0 & \frac{1}{e^2}  & \frac{1}{e^2} & 0   \\
		0 & \frac{1}{e^2} & \frac{1}{e^2} & 0  \\
	     1-\frac{1}{e^2} & 0 & 0 & 1 + \frac{1}{e^2} 
	\end{pmatrix}.
\end{equation*}
This also leads to $\tri^2(\eps) = \eps^\top \bS \eps$.
Defining $S_\delta \coloneqq S_\delta(\eps,P) \coloneqq \frac{P}{2} \frac{\bS \eps}{\trid(\eps)}$, we find $\sigmad(\eps,P) = S_\delta(\eps,P) - \frac{P}{2}\Id_2$.
Hibler's operator is then defined by
\begin{equation*}
    \bAH \vice \coloneqq \frac{1}{\rice h} \divH S_\delta = \frac{1}{\rice h} \divH\left(\frac{P}{2} \frac{\bS \eps}{\sqrt{\delta + \eps^\top \bS \eps}}\right).
\end{equation*}
From \cite[Section~3]{BDHH:22}, we recall its representation
\begin{equation}\label{eq:Hibler op diff form}
    \begin{aligned}
        (\bAH \vice)_i
        &= \sum_{j,k,l=1}^2 \frac{P}{2 \rice h} \frac{1}{\trid(\eps)} \Bigl(\bS_{ij}^{kl} - \frac{1}{\trid^2(\eps)} (\bS \eps)_{ik} (\bS \eps)_{lj}\Bigr) \del_k \eps_{jl} + \frac{1}{2 \rice h \trid(\eps)} \sum_{j=1}^2 (\del_j P) (\bS \eps)_{ij},
    \end{aligned}
\end{equation}
where $i = 1,2$.
For the principal part, it is advantageous to define coefficients by
\begin{equation}\label{eq:coeffs principal part Hibler op}
    a_{ij}^{kl}(\eps,P) \coloneqq -\frac{P}{2 \rice h} \frac{1}{\trid(\eps)} \left(\bS_{ij}^{kl} - \frac{1}{\trid^2(\eps)}(\bS \eps)_{ik} (\bS \eps)_{jl}\right).
\end{equation}
At this stage, it is also worthwhile invoking the linearized Hibler operator.
For $u_0 \in \rC^1(\oOmega)^2 \times \rC(\oOmega) \times \rC(\oOmega)$ such that $h_0 > \kappa$, and with the notation $\rD_m = - \mri \del_m$, it is of the form
\begin{equation}\label{eq:lin Hibler op}
    \begin{aligned}
        [\bAH(u_0) \vice]_i
        &= \sum_{j,k,l=1}^2 a_{ij}^{kl}(\eps(\vii),P(h_0,a_0)) \rD_k \rD_l v_{\ice,j}\\
        &\quad + \frac{1}{2 \rice h_0 \trid(\eps(v_{\ice,0}))} \sum_{j=1}^2 (\del_j P(h_0,a_0)) (\bS \eps(\vice))_{ij}.
    \end{aligned}
\end{equation}
The above conditions on $u_0$ are especially satisfied provided $u_0 \in V$, and $p, q \in (1,\infty)$ fulfill \eqref{eq:cond p and q sea ice para-hyper}.

Concerning the transformed terms, we start with the symmetric part of the gradient, transforming to
\begin{equation*}
    2 \eps_{ij}(\vice) = \del_i v_{\ice,j} + \del_j v_{\ice,i} = \sum_{k=1}^2 (\del_i Y_k) \del_k \tv_{\ice,j} + (\del_j Y_k) \del_k \tv_{\ice,i} \eqqcolon 2 \teps_{ij}(\tvice).
\end{equation*}
For the coefficients as in \eqref{eq:coeffs principal part Hibler op}, we also define
\begin{equation}\label{eq:coeffs transformed Hibler sea ice para-hyper}
    a_{ij}^{klm}(\teps(\tvice),P(\th,\ta)) \coloneqq (\del_k Y_m) a_{ij}^{kl}(\teps(\tvice),P(\th,\ta)).
\end{equation}
Let us observe that the coefficients $a_{ij}^{klm}$ depend on the time $t$ via $\del_k Y_m$.
For simplicity of notation, we do not write this dependence explicitly.
Besides, we compute
\begin{equation*}
    \del_m \teps_{jl}(\vice) = \frac{1}{2} \sum_{n=1}^2\bigl((\del_m \del_j Y_n) \del_n \tv_{ice,l} + (\del_j Y_n) \del_m \del_n \tv_{\ice,l} + (\del_m \del_l Y_n) \del_n \tv_{\ice,j} + (\del_l Y_n) \del_m \del_n \tv_{\ice,j}\bigr).
\end{equation*}
In view of \eqref{eq:Hibler op diff form}, these calculations result in the transformed Hibler operator given by
\begin{equation}\label{eq:transformed Hibler op sea ice para-hyper}
    \begin{aligned}
        \tbAH(\tu) \tvice
        &= \sum_{j,k,l,m=1}^2 a_{ij}^{klm}(\teps(\tvice),P(\th,\ta)) \del_m \teps_{jl}(\tvice)\\
        &\quad + \frac{1}{2 \rice \th \trid(\teps(\tvice))} \sum_{j,k=1}^2 (\del_j Y_k)(\del_k \th + c_\bullet \del_k \ta) (\bS \teps(\tvice))_{ij}.
    \end{aligned}
\end{equation}

The transformed terms associated to the horizontal divergence of the ice strength $P$ read as
\begin{equation*}
    (\tBh(\tu) \th)_i \coloneqq \frac{\del_h P(\th,\ta)}{2 \rice \th} \sum_{j=1}^2 (\del_i Y_j) \del_j \th \tand (\tBa(\tu) \ta)_i \coloneqq \frac{\del_a P(\th,\ta)}{2 \rice \th} \sum_{j=1}^2 (\del_i Y_j) \del_j \ta.
\end{equation*}
For brevity, we will also use the notation $\tB(\tu)\binom{\th}{\ta} \coloneqq \tBh(\tu) \th + \tBa(\tu) \ta$ in the sequel.
The last transformed term involving differential operators is $\divH \vice$, and we get $\divH \vice = \sum_{j,k=1}^2 (\del_j Y_k) \del_k \tv_{\ice,j}$.
The remaining transformed terms are obtained by an insertion of the variables in Lagrangian coordinates instead of Eulerian coordinates, i.\ e., $u = (\vice,h,a)$ is replaced by $(\tvice,\th,\ta)$.
In summary, with the abbreviation $\tice(\tvice) \coloneqq \tatm + \tocn(\tvice)$, the transformed system of equations takes the shape
\begin{equation}\label{eq:transformed parabolic-hyperbolic regularized model}
    \left\{
    \begin{aligned}
        \dt \tvice 
        &= \tbAH(\tu) \tvice - \tB(\tu)\binom{\th}{\ta} - \ccor \tvice^\perp - g \nablaH H + \frac{1}{\rice \th}\tice(\tvice), &&\tin \TOmega,\\
        \dt \th
        &= -\th \sum_{j,k=1}^2 (\del_j Y_k) \del_k \tv_{\ice,j} + \Sh(\th,\ta), &&\tin \TOmega,\\
        \dt \ta
        &= -\ta \sum_{j,k=1}^2 (\del_j Y_k) \del_k \tv_{\ice,j} + \Sa(\th,\ta), &&\tin \TOmega,\\
        \tvice 
        &= 0, &&\ton \TdOmega,\\
        \tvice(0) &= \vii, \enspace \th(0) = h_0, \enspace \ta(0) = a_0, &&\tin \Omega.
    \end{aligned}
    \right.
\end{equation}

We now reformulate the main result in terms of the transformed system.
The first task is then to verify this theorem, and in the end, we will deduce the assertion of \autoref{thm:local strong wp sea ice para-hyper} therefrom.

\begin{thm}\label{thm:local strong wp Lagrangian coords}
Let $p, q \in (1,\infty)$ satisfy \eqref{eq:cond p and q sea ice para-hyper}, let $u_0 \in V$, with $V$ as in \eqref{eq:open set V sea ice para-hyper}, and recall the spaces $\rX_0$ and~$\rX_1$ from \eqref{eq:ground space and regularity space sea ice para-hyper}.
Additionally, suppose that the external terms $\Vatm$, $\Vocn$, $H$ and $\fgr$ fulfill \autoref{ass:external forcing terms}.
Then there exists $T > 0$ such that \eqref{eq:transformed parabolic-hyperbolic regularized model} has a unique solution $\tu = (\tvice,\th,\ta)$ with
\begin{equation*}
    \tu \in \rW^{1,p}(0,T;\rX_0) \cap \rL^p(0,T;\rX_1) \cap \rC([0,T];V).
\end{equation*}
\end{thm}

Next, we study the continuous dependence of the solution on the initial data.
As a by-product, we also obtain the local uniformity of the time of existence $T > 0$ with respect to the initial data.

\begin{coron}\label{cor:cont dep init data}
Under the assumptions of \autoref{thm:local strong wp Lagrangian coords}, let $u_0 \in V$.
Then there are $T > 0$ and $r > 0$ such that $\obB_{\rX_\gamma}(u_0,r) \subset V$, and for all $u_1 \in \obB_{\rX_\gamma}(u_0,r)$, there exists a unique solution $\tu = \tu(\cdot,u_1)$ to \eqref{eq:transformed parabolic-hyperbolic regularized model} on~$(0,T)$ in the regularity class as in \autoref{thm:local strong wp Lagrangian coords}.
Moreover, there exists a constant $C > 0$ such that for all $u_1$, $u_2 \in \obB_{\rX_\gamma}(u_0,r)$, the respective solutions to \eqref{eq:transformed parabolic-hyperbolic regularized model} satisfy $\| \tu(\cdot,u_1) - \tu(\cdot,u_2) \|_{\bE_1} \le C \cdot \| u_1 - u_2 \|_{\rX_\gamma}$.
\end{coron}

We finish this section with a blow-up criterion for the solution $\tu$ in the Lagrangian formulation.

\begin{coron}\label{cor:blow-up crit}
Under the assumptions of \autoref{thm:local strong wp Lagrangian coords}, the solution $\tu$ to \eqref{eq:transformed parabolic-hyperbolic regularized model} exists on a maximal time interval of existence $J(u_0) = [0,t_+(u_0))$, where $t_+(u_0)$ is characterized by
\begin{enumerate}[(a)]
    \item global existence, so $t_+(u_0) = \infty$,
    \item $\liminf_{t \to t_+(u_0)} \dist_{\rX_\gamma}(\tu(t),\del V) = 0$, or
    \item $\lim_{t \to t_+(u_0)} \tu(t)$ does not exist in $\rX_\gamma$.
\end{enumerate}
\end{coron}

\section{Linear theory}
\label{sec:lin theory}

In this section, we address the linear theory.
More precisely, we establish the boundedness of the $\Hinfty$-calculus of the operator matrix associated to the linearization of the parabolic-hyperbolic sea ice model.
This result is of independent interest due to its interesting functional analytic consequences for the operator matrix such as a characterization of its fractional power domains, or stochastic maximal regularity.
We refer here also to the discussion in \autoref{sec:outlook and open problems}.
As a corollary, we especially obtain the maximal $\rL^p$-regularity.
Note that this (weaker) property is sufficient for the fixed point argument.

First, we recall the $\rL^q$-realization of the linearized Hibler operator as well as the bounded $\Hinfty$-calculus that it admits.
For $u_0 \in \rC^{1,\alpha}(\oOmega)^2 \times \rC^{0,\alpha}(\oOmega) \times \rC^{0,\alpha}(\oOmega)$, with $\alpha > 0$ and $h_0 \ge \kappa$, we recall the linearized Hibler operator $\bAH(u_0)$ from \eqref{eq:lin Hibler op}.
The $\rL^q$-realization of the linearized Hibler operator subject to Dirichlet boundary conditions on $\dOmega$ is then given by
\begin{equation}\label{eq:Lq-realization of the Hibler op}
    \left[\AHD(u_0)\right] \vice \coloneqq \left[\bAH(u_0)\right] \vice, \twith \rD(\AHD(u_0)) \coloneqq \rW^{2,q}(\Omega)^2 \cap \rW_0^{1,q}(\Omega)^2.
\end{equation}
The following result can be obtained analogously as in \cite[Section~4]{BDHH:22}.
Let us observe that the reduced regularity in the $h$- and $a$-component does not affect the arguments as the coefficients remain H\"older-continuous, because they depend on $h$ and $a$ in a smooth way.

\begin{lem}\label{lem:Hinfty Hibler op}
Let $q \in (1,\infty)$ as well as $u_0 \in \rC^{1,\alpha}(\oOmega)^2 \times \rC^{0,\alpha}(\oOmega) \times \rC^{0,\alpha}(\oOmega)$ for some $\alpha > 0$ and with $h_0 \ge \kappa$, and consider $\AHD(u_0)$ as defined in \eqref{eq:Lq-realization of the Hibler op}.
Then there exists $\omega_1 \in \bR$ such that $-\AHD(u_0) + \omega \in \Hinfty(\rL^q(\Omega)^2)$ with $\Hinfty$-angle $\phi_{-\AHD(u_0) + \omega}^\infty < \nicefrac{\pi}{2}$ for all $\omega > \omega_1$.    
\end{lem}

Note that the latter lemma especially implies the maximal $\rL^p$-regularity of $-\AHD(u_0) + \omega$, see for example \cite[Section~4.4]{DHP:03}.
In fact, mimicking the procedure in \cite[Section~4]{BDHH:22}, one can show the maximal regularity of the Hibler operator when merely assuming $u_0 \in \rC^1(\oOmega)^2 \times \rC(\oOmega) \times \rC(\oOmega)$.

Next, let $f_1 \colon (0,T) \times \Omega \to \bR^2$ and $f_2$, $f_3 \colon (0,T) \times \Omega \to \bR$.
Besides, we consider suitable initial data~$u_0 = (\vii,h_0,a_0)$ as well as $u_1 \in \rC^{1,\alpha}(\oOmega)^2 \times \rC^{0,\alpha}(\oOmega) \times \rC^{0,\alpha}(\oOmega)$, for $\alpha > 0$, and with $h_1 \ge \kappa$ and~$a_1 \in (0,1)$.
For~$\omega \ge 0$, the linear problem under investigation is given by
\begin{equation}\label{eq:lin sea ice para-hyper}
    \left\{
    \begin{aligned}
        \dt \tvice - (\bAH(u_1) - \omega) \tvice + B_1(u_1) \binom{\th}{\ta}
        &= f_1, &&\tin (0,T) \times \Omega,\\
        \dt \th + h_1 \divH \tvice + \omega \th
        &= f_2, &&\tin (0,T) \times \Omega,\\
        \dt \ta + a_1 \divH \tvice + \omega \ta
        &= f_3, &&\tin (0,T) \times \Omega,\\
        \tvice 
        &= 0, &&\ton (0,T) \times \del \Omega,\\
        \tvice(0) = v_{\ice,0}, \enspace \th(0) = h_0, \enspace \ta(0) 
        &= a_0, &&\tin \Omega,
    \end{aligned}
    \right.
\end{equation}
where $B_1$ is as defined by $B_1(u_1) \binom{\th}{\ta} = \frac{\del_h P(h_1,a_1)}{2 \rice h_1} \nablaH \th + \frac{\del_a P(h_1,a_1)}{2 \rice h_1} \nablaH \ta$.
The operator matrix corresponding to the linearized problem \eqref{eq:lin sea ice para-hyper} for $\omega = 0$ is given by
\begin{equation}\label{eq:op matrix sea ice para-hyper}
    A(u_1) \coloneqq \begin{pmatrix}
        -\AHD(u_1) & \frac{\del_h P(h_1,a_1)}{2 \rice h_1} \nablaH & \frac{\del_a P(h_1,a_1)}{2 \rice h_1} \nablaH\\
        h_1 \divH & 0 & 0\\
        a_1 \divH & 0 & 0
    \end{pmatrix}.
\end{equation}

The main result of this section asserts that, up to a shift, the operator matrix $A(u_1)$ from \eqref{eq:op matrix sea ice para-hyper} admits a bounded $\Hinfty$-calculus on $\rX_0$.
As a preparation, we briefly recall some theory on the $\Hinfty$-calculus for operator matrices.
To this end, we first invoke the notion of a diagonally dominant block operator matrix.

\begin{defn}\label{def:diag dom op matrix}
Consider Banach spaces $\rX$ and $\rY$ as well as linear operators $A \colon \rD(A) \subset \rX \to \rX$, $B \colon \rD(B) \subset \rY \to \rX$, $C \colon \rD(C) \subset \rX \to \rY$ and $D \colon \rD(D) \subset \rY \to \rY$. 
Moreover, set $\rZ \coloneqq \rX \times \rY$.
If 
\begin{enumerate}[(a)]
    \item the operators $A$ and $D$ are closed, linear and densely defined, and
    \item it holds that $C$ is relatively $A$-bounded and $B$ is relatively $D$-bounded, i.\ e., $\rD(D) \subset \rD(B)$, $\rD(A) \subset \rD(C)$, and there are constants $c_A$, $c_D$, $C \ge 0$ with
    \begin{equation*}
        \begin{aligned}
            \| C x \|_{\rX} 
            &\le c_A \cdot \| A x \|_{\rX} + L \cdot \| x \|_{\rX} \tforall x \in \rD(A), \tand\\
            \| B y \|_{\rX} 
            &\le c_D \cdot \| D y \|_{\rY} + L \cdot \| y \|_{\rY} \tforall y \in \rD(D),
        \end{aligned}
    \end{equation*}
\end{enumerate}
then the operator matrix $\cA \colon \rD(\cA) \coloneqq \rD(A) \times \rD(D) \subset \rZ \to \rZ$ defined by
\begin{equation}\label{eq:block op matrix}
    \cA \binom{x}{y} \coloneqq \begin{pmatrix}
        A & B\\
        C & D
    \end{pmatrix} \binom{x}{y}, \tfor \binom{x}{y} \in \rD(\cA),
\end{equation}
is called \emph{diagonally dominant}.
\end{defn}

The following lemma tailored to our setting is a direct consequence of \cite[Corollary~7.2]{AH:23}.

\begin{lem}\label{lem:Hinfty diag dom op matrix}
Let $\cA$ as defined in \eqref{eq:block op matrix} be diagonally dominant in the sense of \autoref{def:diag dom op matrix}, and assume that $A$ is sectorial on $\rX$ with spectral angle $\phi_A \in [0,\pi)$ as well as $D \in \cL(\rY)$.   
If in addition, $A \in \Hinfty(\rX)$ with $\phi_A^\infty \in [0,\pi)$, then for every $\phi \in (\phi_A^\infty,\pi)$, there is $\omega \ge 0$ such that $\cA + \omega \in \Hinfty(\rZ)$ and $\phi_{\cA + \omega}^\infty \le \phi$.
\end{lem}
Let us note that the diagonal dominance of $\cA$ together with the assumption $D \in \cL(\rY)$ already requires the boundedness of the operator $B$, i.\ e., $B \in \cL(\rY,\rX)$.

We are now in the position to address the main result of this section on the $\Hinfty$-calculus of $A(u_1)$.

\begin{thm}\label{thm:Hinfty sea ice para-hyper}
Let $q \in (2,\infty)$ and $u_1 \in \rC^{1,\alpha}(\oOmega)^2 \times \rW^{1,q}(\Omega) \times \rW^{1,q}(\Omega)$ for some $\alpha > 0$ and with $h_1 \ge \kappa$.
Then there is $\omega_0 \in \bR$ such that for the operator matrix $A(u_1)$ from \eqref{eq:op matrix sea ice para-hyper}, it holds that $A(u_1) + \omega \in \Hinfty(\rX_0)$ with $\phi_{A(u_1) + \omega}^\infty < \nicefrac{\pi}{2}$ for all $\omega > \omega_0$.
\end{thm}

\begin{proof}
The idea of the proof is to make use of \autoref{lem:Hinfty diag dom op matrix}.
At this stage, the choice of the anisotropic ground space $\rX_0$ as introduced in \eqref{eq:ground space and regularity space sea ice para-hyper} is crucial in order to establish the diagonal dominance of $A(u_1)$.

For this purpose, we recall from \autoref{lem:Hinfty Hibler op} that for $\omega > \omega_1$, we have $-\AHD(u_1) + \omega \in \Hinfty(\rL^q(\Omega)^2)$ with $\phi_{-\AHD(u_1) + \omega}^\infty < \nicefrac{\pi}{2}$ upon observing that $\rW^{1,q}(\Omega) \hookrightarrow \rC^{0,\beta}(\oOmega)$ for some $\beta > 0$ thanks to $q > 2$.
On the other hand, for such $\omega$, the operator $\diag(\omega \Id,\omega \Id)$ is clearly bounded on $\rW^{1,q}(\Omega) \times \rW^{1,q}(\Omega)$. 
It thus remains to handle the off-diagonal terms.
With regard to the remark following \autoref{lem:Hinfty diag dom op matrix}, we have to show the boundedness of the terms above the diagonal.
The shape of $P(h_1,a_1)$ from \eqref{eq:ice strength} as well as the assumptions on $u_1$ directly yield the existence of $C > 0$ such that
\begin{equation*}
    \left\| \frac{\del_h P(h_1,a_1)}{2 \rice h_1} \nablaH h \right\|_{\rL^q(\Omega)} \le C \cdot \| h \|_{\rW^{1,q}(\Omega)}, \tfor h \in \rW^{1,q}(\Omega),
\end{equation*}
and likewise for the corresponding $a$-term.
Next, let $\vice \in \rD(\AHD(u_1)) = \rW^{2,q}(\Omega)^2 \cap \rW_0^{1,q}(\Omega)$.
Employing the Banach algebra structure of $\rW^{1,q}(\Omega)$ by $q > 2$, and using the equivalence of the norm of $\rW^{2,q}(\Omega)^2$ with the graph norm of $-\AHD(u_1) + \omega$ for $\omega > 0$ sufficiently large, we find that
\begin{equation*}
    \| h_1 \divH \vice \|_{\rW^{1,q}(\Omega)} \le C \cdot \| \vice \|_{\rW^{2,q}(\Omega)} \le C \cdot \| (-\AHD(u_1) + \omega) \vice \|_{\rL^q(\Omega)}.
\end{equation*}
The other term below the diagonal can be handled analogously.
In total, we infer that the operator matrix $A(u_1) + \omega$, with $A(u_1)$ from \eqref{eq:op matrix sea ice para-hyper}, is diagonally dominant.
Thus, for $\omega_0 \ge \omega_1$, we deduce the assertion of the theorem from \autoref{lem:Hinfty diag dom op matrix}.
\end{proof}

Having proved the $\Hinfty$-calculus of the operator matrix associated to the linearization, we now focus on the maximal $\rL^p$-regularity of the linearized problem \eqref{eq:lin sea ice para-hyper}.
To this end, let us briefly introduce the so-called data space $\bE_0$ and the maximal regularity space $\bE_1$.
With $\rX_0$ and~$\rX_1$ from \eqref{eq:ground space and regularity space sea ice para-hyper}, we set
\begin{equation}\label{eq:data and max reg space sea ice para-hyper}
    \bE_0 \coloneqq \rL^p(0,T;\rX_0) \tand \bE_1 \coloneqq \rW^{1,p}(0,T;\rX_0) \cap \rL^p(0,T;\rX_1).
\end{equation}
The maximal regularity result below is a consequence of \autoref{thm:Hinfty sea ice para-hyper}, since a bounded $\Hinfty$-calculus with angle smaller than $\nicefrac{\pi}{2}$ implies the maximal $\rL^p$-regularity on UMD spaces, see e.\ g.\ \cite[Section~4]{DHP:03}, and the present ground space~$\rX_0$ from \eqref{eq:ground space and regularity space sea ice para-hyper} is in particular a UMD space.

\begin{coron}\label{cor:max reg sea ice para-hyper}
Let $T \in (0,\infty]$, $p \in (1,\infty)$, $q \in (2,\infty)$ with $\nicefrac{2}{p} + \nicefrac{1}{q} \neq 2$, and for $u_1 = (v_{\ice,1},h_1,a_1)$, assume $u_1 \in \rC^{1,\alpha}(\oOmega)^2 \times \rW^{1,q}(\Omega) \times \rW^{1,q}(\Omega)$, with $\alpha > 0$, $h_1 \ge \kappa$ and $a_1 \in (0,1)$.
Furthermore, for the time trace space $\rX_\gamma$ and the data space $\bE_0$ as in \eqref{eq:trace space sea ice para-hyper} and \eqref{eq:data and max reg space sea ice para-hyper}, consider $u_0 = (\vii,h_0,a_0) \in \rX_\gamma$ and~$(f_1,f_2,f_3) \in \bE_0$.
Then there exists $\omega_0 \in \bR$ such that for all $\omega > \omega_0$, the linearized problem \eqref{eq:lin sea ice para-hyper} has a unique solution $\tu = (\tvice,\th,\ta) \in \bE_1$, with $\bE_1$ as introduced in \eqref{eq:data and max reg space sea ice para-hyper}.
Besides, there is a constant~$\Cmr > 0$ so that the unique solution $\tu$ of \eqref{eq:lin sea ice para-hyper} satisfies
\begin{equation*}
    \| \tu \|_{\bE_1} \le \Cmr\left(\| (f_1,f_2,f_3) \|_{\bE_0} + \| u_0 \|_{\rX_\gamma}\right).
\end{equation*}
\end{coron}

\begin{rem}
The assertion of \autoref{cor:max reg sea ice para-hyper} can also be obtained by a direct perturbation argument when only assuming $u_1 \in \rC^1(\oOmega)^2 \times \rW^{1,q}(\Omega) \times \rW^{1,q}(\Omega)$, see also \cite[Proposition~6.2.1]{Bra:24}.
\end{rem}

In the sequel, we fix $\omega > \omega_0$ for $\omega_0 > 0$ from \autoref{cor:max reg sea ice para-hyper}.
Note that the maximal regularity constant~$\Cmr > 0$ from \autoref{cor:max reg sea ice para-hyper} generally depends on $T$.
However, it can be chosen independent of~$T > 0$ provided homogeneous initial values are considered.
In order reduce the considerations to this situation, we now discuss the so-called reference solution capturing the initial values.
More precisely, for~$u_0 \in V$, the resulting system reads as
\begin{equation}\label{eq:lin for ref sol sea ice para-hyper}
    \left\{
    \begin{aligned}
        \dt \tvice - (\bAH(u_0) - \omega) \tvice + B_1(u_0) \binom{\th}{\ta}
        &= 0, &&\tin (0,T) \times \Omega,\\
        \dt \th + h_0 \divH \tvice + \omega \th
        &= 0, &&\tin (0,T) \times \Omega,\\
        \dt \ta + a_0 \divH \tvice + \omega \ta
        &= 0, &&\tin (0,T) \times \Omega,\\
        \tvice 
        &= 0, &&\ton (0,T) \times \del \Omega,\\
        \tvice(0) = v_{\ice,0}, \enspace \th(0) = h_0, \enspace \ta(0) 
        &= a_0, &&\tin \Omega.
    \end{aligned}
    \right.
\end{equation}

The result below on the existence of a unique reference solution to \eqref{eq:lin for ref sol sea ice para-hyper} follows from \autoref{cor:max reg sea ice para-hyper}.

\begin{prop}\label{prop:ref sol sea ice para-hyper}
Let $p, q \in (1,\infty)$ satisfy \eqref{eq:cond p and q sea ice para-hyper}, and consider $0 < T \le \infty$ and $u_0 = (\vii,h_0,a_0) \in V$, with $V$ from \eqref{eq:open set V sea ice para-hyper}.
Then \eqref{eq:lin for ref sol sea ice para-hyper} admits a unique solution $u_0^* = (\vii^*,h_0^*,a_0^*) \in \bE_1$, with $\bE_1$ from \eqref{eq:data and max reg space sea ice para-hyper}.
\end{prop}

The importance of the reference solution $u_0^*$ for the arguments later on calls for additional remarks.

\begin{rem}\label{rem:ref sol sea ice para-hyper}
\begin{enumerate}[(a)]
    \item The reference solution $u_0^*$ results from an application of the semigroup generated by~$A(u_0) + \omega$, with $A(u_0)$ as defined in \eqref{eq:op matrix sea ice para-hyper}, to the initial values $u_0$, so $u_0^*(t) = \mre^{-(A(u_0) + \omega)(t)} u_0$.
    \item The norm of $u_0^*$ will be denoted by $C_T^* \coloneqq \| u_0^* \|_{\bE_1}$, and we observe that $C_T^* \to 0$ as $T \to 0$.
    \item From~(a), it follows that $\| u_0^* - u_0 \|_{\rBUC([0,T];\rX_\gamma)} \to 0$ as $T \to 0$.
    In particular, as $V \subset \rX_\gamma$ is open, there is a sufficiently small $r_0 > 0$ with $\obB_{\rX_\gamma}(u_0,r_0) \subset V$ for given $u_0 \in V$.
    Thus, $u_0^*(t) \in V$ holds for all $t \in [0,T_0]$, for $T_0 > 0$ sufficiently small, since
    \begin{equation}\label{eq:ref sol close to init data sea ice para-hyper}
        \sup_{t \in [0,T_0]} \| u_0^*(t) - u_0 \|_{\rX_\gamma} \le \frac{r_0}{2}.
    \end{equation}
\end{enumerate}
\end{rem}

\section{Proof of the main result}\label{sec:proof main result}

This final section presents the proof of the main result.
For this, we first elaborate on the fixed point argument, then provide the required estimates and conclude the proof of \autoref{thm:local strong wp Lagrangian coords} by the contraction mapping principle.
The final step is to deduce the proof of \autoref{thm:local strong wp sea ice para-hyper} from there.

\subsection{Reformulation as a fixed point problem}\label{ssec:fixed point probl}
\

The aim of this subsection is to reformulate the task of finding a solution to the transformed system~\eqref{eq:transformed parabolic-hyperbolic regularized model} as a fixed point problem.
Thus, consider a solution $\tu = (\tvice,\th,\ta)$ to \eqref{eq:transformed parabolic-hyperbolic regularized model}, recall the reference solution $u_0^* = (\vii^*,h_0^*,a_0^*)$ from \autoref{prop:ref sol sea ice para-hyper} and define $\hu = (\hvice,\hh,\ha)$ by $\hu \coloneqq \tu - u_0^*$.
Then $\hu$ solves
\begin{equation}\label{eq:linearized eq for fixed point sea ice para-hyper}
    \left\{
    \begin{aligned}
        \dt \hvice - (\bAH(u_0) - \omega) \hvice + B_1(u_0) \binom{\hh}{\ha}
        &= F_1(\hu), &&\tin (0,T) \times \Omega,\\
        \dt \hh + h_0 \divH \hvice + \omega \hh
        &= F_2(\hu), &&\tin (0,T) \times \Omega,\\
        \dt \ha + a_0 \divH \hvice + \omega \ha
        &= F_3(\hu), &&\tin (0,T) \times \Omega,\\
        \hvice 
        &= 0, &&\ton (0,T) \times \del \Omega,\\
        \hvice(0) = 0, \enspace \hh(0) = 0, \enspace \ha(0) 
        &= 0, &&\tin \Omega,
    \end{aligned}
    \right.
\end{equation}
for the right-hand sides
\begin{equation}\label{eq:right-hand sides sea ice para-hyper}
    \begin{aligned}
        F_1(\hu)
        &\coloneqq \bigl(\tbAH(\tu) - \bAH(u_0)\bigr)\tvice - \bigl(\tB(\tu) - B_1(u_0)\bigr)\binom{\th}{\ta} + \omega \tvice - \ccor \tvice^\perp - g \nablaH H\\
        &\quad + \frac{1}{\rice \th}\bigl(\tatm + \tocn(\tvice)\bigr),\\
        F_2(\hu)
        &\coloneqq h_0 \divH \tvice - \th \sum_{j,k=1}^2 (\del_j Y_k) \del_k \tv_{\ice,j} + \omega \th + \Sh(\th,\ta) \tand\\
        F_3(\hu)
        &\coloneqq a_0 \divH \tvice - \ta \sum_{j,k=1}^2 (\del_j Y_k) \del_k \tv_{\ice,j} + \omega \ta + \Sa(\th,\ta).
    \end{aligned}
\end{equation}

In the following, we denote by $\prescript{}{0}{\bE_1}$ the elements in the maximal regularity space $\bE_1$ from \eqref{eq:data and max reg space sea ice para-hyper} with homogeneous initial values, so $\hu \in \prescript{}{0}{\bE_1}$ fulfills $\hu(0) = 0$.
With regard to \eqref{eq:linearized eq for fixed point sea ice para-hyper}, the solution to the transformed problem $\tu$ results from the addition of $\hu \in \bE_1$ and the reference solution $u_0^*$.
We demand that this solution also only takes values in the physically relevant range, so it should hold that $\tu(t) \in V$ for all $t \in [0,T_1]$, where $T_1 > 0$ will be determined in the sequel.
Recall that given $u_0 \in V$, there exists~$r_0 > 0$ so that $\obB_{\rX_\gamma}(u_0,r_0) \subset V$.
For $T_0 > 0$ as in \autoref{rem:ref sol sea ice para-hyper}, we especially have $u_0^*(t) \in V$ for all~$t \in [0,T_0]$.
Moreover, we deduce from \cite[Theorem~III.4.10.2]{Ama:95} the embedding
\begin{equation}\label{eq:emb max reg space into BUC Xgamma sea ice para-hyper}
    \prescript{}{0}{\bE_1} \hookrightarrow \rBUC([0,T_0];\rX_\gamma).
\end{equation}
In the above, $\rBUC([0,T_0];\rX_\gamma)$ represents the bounded and uniformly continuous functions in time with values in $\rX_\gamma$.
Combining \eqref{eq:ref sol close to init data sea ice para-hyper} and the choice $R_0 \le \nicefrac{r_0}{2 C}$ for the embedding constant $C > 0$ from \eqref{eq:emb max reg space into BUC Xgamma sea ice para-hyper}, we conclude for $\hu \in \prescript{}{0}{\bE_1}$ with $\| \hu \|_{\bE_1} \le R_0$ that
\begin{equation*}
    \sup_{t \in [0,T_0]} \| \tu(t) - u_0 \|_{\rX_\gamma} \le C \cdot \| \hu \|_{\bE_1} + \sup_{t \in [0,T_0]} \| u_0^*(t) - u_0 \|_{\rX_\gamma} \le \frac{r_0}{2} + \frac{r_0}{2} \le r_0.
\end{equation*}

The preceding discussion can be summarized in the lemma below.

\begin{lem}\label{lem:final sol contained open set V sea ice para-hyper}
Let $u_0 \in V$, $T_0 > 0$ as introduced in \autoref{rem:ref sol sea ice para-hyper} and $0 < R_0 \le \nicefrac{r_0}{2 C}$, with $C > 0$ representing the embedding constant from \eqref{eq:emb max reg space into BUC Xgamma sea ice para-hyper}, and $r_0 > 0$ is such that $\obB_{\rX_\gamma}(u_0,r_0) \subset V$.
Consider $T \in (0,T_0]$, $R \in (0,R_0]$ and $\tu \coloneqq \hu + u_0^*$, where $\hu \in \prescript{}{0}{\bE_1}$ with $\| \hu \|_{\bE_1} \le R$, and $u_0^* \in \bE_1$ denotes the reference solution from \autoref{prop:ref sol sea ice para-hyper}.
Then $\tu(t) \in V$ for all $t \in [0,T]$.
\end{lem}

Next, we describe the fixed point argument in more details.
Let $p, q \in (1,\infty)$ be such that \eqref{eq:cond p and q sea ice para-hyper} is valid, and consider $T_0 > 0$ and $R_0 > 0$ from \autoref{lem:final sol contained open set V sea ice para-hyper}.
For $T \in (0,T_0]$ and $R \in (0,R_0]$, we then define
\begin{equation}\label{eq:set and map fixed point arg sea ice para-hyper}
    \cK_T^R \coloneqq \{\overline{u} \in \prescript{}{0}{\bE_1} : \| \overline{u} \|_{\bE_1} \le R\} \tand \Phi_T^R \colon \cK_T^R \to \prescript{}{0}{\bE_1}, \twith \Phi_T^R(\overline{u}) \coloneqq \hu,
\end{equation}
where $\hu$ represents the unique solution to \eqref{eq:linearized eq for fixed point sea ice para-hyper} with right-hand sides $F_1(\overline{u})$, $F_2(\overline{u})$ and $F_3(\overline{u})$ as introduced in \eqref{eq:right-hand sides sea ice para-hyper}.
By the maximal regularity result, \autoref{cor:max reg sea ice para-hyper}, the map $\Phi_T^R$ is well-defined provided the terms on the right-hand side lie in the data space $\bE_0$ as introduced in \eqref{eq:data and max reg space sea ice para-hyper}.

\subsection{Estimates of the nonlinear terms}\label{ssec:ests of nonlin terms}
\

This subsection presents the estimates of the nonlinear terms $F_1$, $F_2$ and $F_3$.
For this, we require some further preparation, and we start by discussing properties of the transformation to Lagrangian coordinates.
Recall that
\begin{equation}\label{eq:diffeo X Euler Lagrange sea ice para-hyper}
    X(t,\yH) = \yH + \int_0^t \tvice(s,\yH) \srd s.
\end{equation}
Moreover, for the inverse $Y(t,\cdot)$ of $X(t,\cdot)$, we have $\nablaH Y(t,X(t,\yH)) = [\nablaH X]^{-1}(t,\yH)$.
Aiming to track the dependence of the diffeomorphisms $X$ and $Y$ on $\tvice$, we also use the notation $X_{\tvice}$ and $Y_{\tvice}$.
In the sequel, $\tvice$ results from the addition of $\hu \in \prescript{}{0}{\bE_1}$ and the reference solution $u_0^*$ from \autoref{prop:ref sol sea ice para-hyper}.
We also introduce the notation $\bE_1 = \bE_1^v \times \bE_1^h \times \bE_1^a$, and we conclude that
\begin{equation}\label{eq:bound of the norm tu sea ice para-hyper}
    \| \tu \|_{\bE_1} \le R + C_T^* \le R_0 + C_{T_0}^*, \tso \| \tvice \|_{\bE_1^v} \le R_0 + C_{T_0}^*.
\end{equation}
Given $X = X_{\tvice}$ as in \eqref{eq:diffeo X Euler Lagrange sea ice para-hyper} and $p' \in (1,\infty)$ representing the H\"older conjugate of $p$, we deduce from H\"older's inequality and \eqref{eq:bound of the norm tu sea ice para-hyper} that
\begin{equation}\label{eq:est of nablaH X - Id sea ice para-hyper}
    \sup_{t\in [0,T]} \| \nablaH X_{\tvice} - \Id_2 \|_{\rW^{1,q}(\Omega)} \le C \int_0^T \| \nablaH \tvice(t,\cdot) \|_{\rW^{1,q}(\Omega)} \le C T^{\nicefrac{1}{p'}} (R_0 + C_{T_0}^*).
\end{equation}
By virtue of the embedding $\rW^{1,q}(\Omega) \hookrightarrow \rL^\infty(\Omega)$, thanks to $q > 2$, we infer that there is $T_0' > 0$ small with
\begin{equation}\label{eq:est of nablaHX - Id in Linfty}
    \sup_{t \in [0,T_0']} \| \nablaH X_{\tvice} - \Id_2 \|_{\rL^\infty(\Omega)} \le \frac{1}{2}.
\end{equation}
As a result, $\nablaH X_{\tvice}(t,\cdot)$ is invertible for all $t \in [0,T_0']$, and $\nablaH Y_{\tvice}(t,\cdot)$ thus exists on $[0,T_0']$.
It also follows from $\dt \nablaH X_{\tvice}(t,\yH) = \nablaH \tvice(t,\yH)$ and \eqref{eq:bound of the norm tu sea ice para-hyper} that
\begin{equation*}
    \| \dt \nablaH X_{\tvice}(t,\cdot) \|_{\rL^p(0,T;\rW^{1,q}(\Omega))} \le C \cdot \| \tvice \|_{\bE_1^v} \le C (R_0 + C_{T_0}^*).
\end{equation*}
Similarly, for $\tu_i = \hu_i + u_0^*$, with $\hu_i \in \prescript{}{0}{\bE_1}$ and $u_0^*$ representing the reference solution from \autoref{prop:ref sol sea ice para-hyper}, we denote the diffeomorphisms related to $\tv_{\ice,i}$ by $X^{(i)}$ and $Y^{(i)}$ for $i=1,2$.
As in \eqref{eq:est of nablaH X - Id sea ice para-hyper}, it follows that
\begin{equation*}
    \sup_{t \in [0,T]} \| \nablaH X^{(1)} - \nablaH X^{(2)} \|_{\rL^\infty(\Omega)} + \sup_{t \in [0,T]} \| \nablaH X^{(1)} - \nablaH X^{(2)} \|_{\rW^{1,q}(\Omega)} \le C T^{\nicefrac{1}{p'}} \cdot \| \hu_1 - \hu_2 \|_{\bE_1}.
\end{equation*}
In addition, the identity $\dt (\nablaH X^{(1)} - \nablaH X^{(2)}) = \nablaH (\hv_{\ice,1} - \hv_{\ice,2})$ leads to
\begin{equation*}
    \| \dt \nablaH(X^{(1)} - X^{(2)}) \|_{\rL^p(0,T;\rW^{1,q}(\Omega))} \le C \cdot \| \hu_1 - \hu_2 \|_{\bE_1}.
\end{equation*}

We collect these considerations in the lemma below, where we additionally include estimates of $\nablaH Y$.

\begin{lem}\label{lem:props of the trafo Euler Lagrange sea ice para-hyper}
Let $p, q \in (1,\infty)$ satisfy \eqref{eq:cond p and q sea ice para-hyper}, and consider $T \in (0,T_1]$, where $T_1 \coloneqq \min\{T_0,T_0'\} > 0$, for $T_0 > 0$ chosen in \eqref{eq:ref sol close to init data sea ice para-hyper} and $T_0' > 0$ related to \eqref{eq:est of nablaHX - Id in Linfty}.
Moreover, let $R \in (0,R_0]$, with $R_0 > 0$ as in \autoref{lem:final sol contained open set V sea ice para-hyper}, and recall the reference solution $u_0^*$ from \autoref{prop:ref sol sea ice para-hyper}.
\begin{enumerate}[(a)]
    \item Consider $\tu \coloneqq \hu + u_0^*$, where $\hu \in \prescript{}{0}{\bE_1}$ with $\| \hu \|_{\bE_1} \le R$.
    Then for $X_{\tvice}$ as in \eqref{eq:diffeo X Euler Lagrange sea ice para-hyper} as well as~$\nablaH Y_{\tvice} \coloneqq [\nablaH X_{\tvice}]^{-1}$, whose existence is guaranteed by \eqref{eq:est of nablaHX - Id in Linfty}, for some $C > 0$, we get
    \begin{equation*}
        \begin{aligned}
            \| \nablaH X_{\tvice} \|_{\rW^{1,p}(0,T;\rW^{1,q}(\Omega))} + \| \nablaH X_{\tvice} \|_{\rL^\infty(0,T;\rW^{1,q}(\Omega))} 
            &\le C \tand\\
            \| \nablaH Y_{\tvice} \|_{\rW^{1,p}(0,T;\rW^{1,q}(\Omega))} + \| \nablaH Y_{\tvice} \|_{\rL^\infty(0,T;\rW^{1,q}(\Omega))} 
            &\le C.
        \end{aligned}
    \end{equation*}
    \item Let $\tu_i = \hu_i + u_0^*$, with $\hu_i \in \prescript{}{0}{\bE_1}$ and $\| \hu_i \|_{\bE_1} \le R$, $i=1,2$.
    Then for $X^{(1)}$, $X^{(2)}$, $Y^{(1)}$ and $Y^{(2)}$, there is a constant $C > 0$ such that
    \begin{equation*}
        \begin{aligned}
            \| \nablaH (X^{(1)} - X^{(2)}) \|_{\rW^{1,p}(0,T;\rW^{1,q}(\Omega))}+ \| \nablaH (X^{(1)} - X^{(2)}) \|_{\rL^\infty(0,T;\rW^{1,q}(\Omega))}
            &\le C \cdot \| \hu_1 - \hu_2 \|_{\bE_1} \tand\\
            \| \nablaH (Y^{(1)} - Y^{(2)}) \|_{\rW^{1,p}(0,T;\rW^{1,q}(\Omega))} + \| \nablaH (Y^{(1)} - Y^{(2)}) \|_{\rL^\infty(0,T;\rW^{1,q}(\Omega))} 
            &\le C \cdot \| \hu_1 - \hu_2 \|_{\bE_1}.
        \end{aligned}
    \end{equation*}
\end{enumerate}
\end{lem}

\begin{proof}
The estimates of $\nablaH X_{\tvice}$ follow from the above considerations.
Thanks to \eqref{eq:cond p and q sea ice para-hyper}, the spaces $\rW^{1,p}(0,T;\rW^{1,q}(\Omega))$ and $\rL^\infty(0,T;\rW^{1,q}(\Omega))$ have the Banach algebra structure, so the norm of products can be estimated by the product of the norms.
From the estimates of $\nablaH X_{\tvice}$, it then follows that
\begin{equation*}
    \begin{aligned}
        \| \det \nablaH X_{\tvice} \|_{\rW^{1,p}(0,T;\rW^{1,q}(\Omega))} + \| \det \nablaH X_{\tvice} \|_{\rL^\infty(0,T;\rW^{1,q}(\Omega))} 
        &\le C \tand\\
        \| \Cof \nablaH X_{\tvice} \|_{\rW^{1,p}(0,T;\rW^{1,q}(\Omega))} + \| \Cof \nablaH X_{\tvice} \|_{\rL^\infty(0,T;\rW^{1,q}(\Omega))} 
        &\le C.
    \end{aligned}
\end{equation*}
With regard to \eqref{eq:est of nablaHX - Id in Linfty}, we have $\det \nablaH X_{\tvice} \ge C > 0$ on $\TOmega$ if $T \le T_0'$, so the representation
\begin{equation*}
    \nablaH Y_{\tvice} = [\nablaH X_{\tvice}]^{-1} = \frac{1}{\det \nablaH X_{\tvice}} (\Cof \nablaH X_{\tvice})^\top
\end{equation*}
leads to the desired estimates of $\nablaH Y_{\tvice}$, showing~(a).
Concerning~(b), we also remark that the estimates of $\nablaH (X^{(1)} - X^{(2)})$ follow from the arguments given before this lemma, while for the estimates of~$\nablaH (Y^{(1)} - Y^{(2)})$, we can proceed in a similar way as for~(a).
In fact, we use the relation
\begin{equation*}
    \nablaH(Y^{(1)} - Y^{(2)}) = \nablaH Y^{(1)} - \nablaH Y^{(2)} = - \nablaH Y^{(1)} (\nablaH(X^{(1)} - X^{(2)})) \nablaH Y^{(2)}.
\end{equation*}
Finally, the estimates follow by a concatenation of~(a) and the estimates of $\nablaH(X^{(1)} - X^{(2)})$ in~(b).
\end{proof}

Before estimating the nonlinear terms, we collect some further useful embeddings and relations.
The assertion of~(a) follows by \cite[Theorem~III.4.10.2]{Ama:95}, where the $T$-independence results from the existence of an extension operator with $T$-independent norm thanks to the homogeneous initial values.
Aspect~(b) is a direct consequence of Sobolev embeddings and \eqref{eq:emd trace space sea ice para-hyper}, while~(c) is implied by H\"older's inequality. 

\begin{lem}\label{lem:aux results sea ice para-hyper}
Let $p, q \in (1,\infty)$ be such that \eqref{eq:cond p and q sea ice para-hyper} is valid, consider $T \in (0,T_1]$, where $T_1 > 0$ is as made precise in \autoref{lem:props of the trafo Euler Lagrange sea ice para-hyper}, and recall the trace space $\rX_\gamma = \rX_\gamma^v \times \rX_\gamma^h \times \rX_\gamma^a$ from \eqref{eq:trace space sea ice para-hyper}.
\begin{enumerate}[(a)]
    \item We have $\prescript{}{0}{\bE_1} \hookrightarrow \rBUC([0,T];\rX_\gamma)$, with $T$-independent embedding constant.
    \item It holds that $\rBUC([0,T];\rX_\gamma^v) \hookrightarrow \rL^\infty(0,T;\rL^{2q}(\Omega)^2) \hookrightarrow \rL^{2p}(0,T;\rL^{2q}(\Omega)^2)$.
    In particular, for $T$-independent embedding constant, we obtain $\prescript{}{0}{\bE_1^v} \hookrightarrow \rBUC([0,T];\rX_\gamma^v) \hookrightarrow \rL^{2p}(0,T;\rL^{2q}(\Omega)^2)$.
    \item For $f \in \rL^\infty(0,T)$, we get $\| f \|_{\rL^p(0,T)} \le T^{\nicefrac{1}{p}} \cdot \| f \|_{\rL^\infty(0,T)}$, while for $f \in \rW^{1,p}(0,T)$ with $f(0) = 0$, we get $\| f \|_{\rL^\infty(0,T)} \le T^{\nicefrac{1}{p'}} \cdot \| f \|_{\rW^{1,p}(0,T)}$, and $p' \in (1,\infty)$ denotes again the H\"older conjugate of $p$.
\end{enumerate}
\end{lem}

A combination of \autoref{lem:props of the trafo Euler Lagrange sea ice para-hyper} and \autoref{lem:aux results sea ice para-hyper}(c) leads to the following estimate of the difference of~$\nablaH Y$ from $\Id_2$ as well as the difference $\nablaH (Y^{(1)} - Y^{(2)})$.

\begin{lem}\label{lem:difference est of Id and nablaH Y}
Let $p,q \in (1,\infty)$ fulfill \eqref{eq:cond p and q sea ice para-hyper}, let $T \in (0,T_1]$ and $R \in (0,R_0]$, with $T_1 > 0$ from \autoref{lem:props of the trafo Euler Lagrange sea ice para-hyper} and $R_0 > 0$ as in \autoref{lem:final sol contained open set V sea ice para-hyper}.
Moreover, recall the reference solution $u_0^*$ from \autoref{prop:ref sol sea ice para-hyper}.
\begin{enumerate}[(a)]
    \item Let $\tu \coloneqq \hu + u_0^*$ for $\hu \in \prescript{}{0}{\bE_1}$ with $\| \hu \|_{\bE_1} \le R$.
    Then for $X_{\tvice}$ from \eqref{eq:diffeo X Euler Lagrange sea ice para-hyper}, $\nablaH Y_{\tvice} = [\nablaH X_{\vice}]^{-1}$ and the H\"older conjugate $p' \in (1,\infty)$ of $p$, we have $\| \Id_2 - \nablaH Y_{\tvice} \|_{\rL^\infty(0,T;\rW^{1,q}(\Omega))} \le C T^{\nicefrac{1}{p'}}$.
    \item For $i=1,2$, consider $\tu_i = \hu_i + u_0^*$, where $\hu_i \in \prescript{}{0}{\bE_1}$ and $\| \hu_i \|_{\bE_1} \le R$.
    Then for a constant $C > 0$ and the H\"older conjugate $p' \in (1,\infty)$ of $p$, the maps $Y^{(1)}$ and $Y^{(2)}$ satisfy the estimate $\| \nablaH (Y^{(1)} - Y^{(2)}) \|_{\rL^\infty(0,T;\rW^{1,q}(\Omega))} \le C T^{\nicefrac{1}{p'}} \cdot \| \hu_1 - \hu_2 \|_{\bE_1}$.
\end{enumerate}
\end{lem}

Having the analytical tools at hand, we now start with the actual estimates of the nonlinear terms.
In the sequel, we fix $T_1 > 0$ from \autoref{lem:props of the trafo Euler Lagrange sea ice para-hyper} and $R_0 > 0$ from \autoref{lem:final sol contained open set V sea ice para-hyper}, and we consider $T \in (0,T_1]$ and~$R \in (0,R_0]$.
We begin with the quasilinear terms from the non-transformed equation.

\begin{lem}\label{lem:ests of auton terms sea ice para-hyper}
Let $p,q \in (1,\infty)$ satisfy \eqref{eq:cond p and q sea ice para-hyper}, and consider $\tu \coloneqq \hu + u_0^*$ for $\hu \in \prescript{}{0}{\bE_1}$ with $\| \hu \|_{\bE_1} \le R$ and~$u_0^*$ denoting the reference solution from \autoref{prop:ref sol sea ice para-hyper}.
Then there exists a constant $C > 0$ with
\begin{equation*}
    \biggl\| \bigl(\bAH(u_0^*) - \bAH(u_0)\bigr)\tvice - \bigl(B_1(u_0^*) - B_1(u_0)\bigr)\binom{\th}{\ta}\biggr\|_{\rL^p(0,T;\rL^q(\Omega))} \le C \cdot \| u_0^* - u_0 \|_{\rBUC([0,T];\rX_\gamma)} (R+C_T^*)
\end{equation*}
and
\begin{equation*}
    \biggl\| \bigl(\bAH(\tu) - \bAH(u_0^*)\bigr)(\tvice) - \bigl(B_1(\tu) - B_1(u_0^*)\bigr)\binom{\th}{\ta}\biggr\|_{\rL^p(0,T;\rL^q(\Omega))} \le C R (R + C_T^*).
\end{equation*}
\end{lem}

\begin{proof}
Let $u_1, u_2 \in V$ and $w = (\vice,h,a) \in \rX_1$.
Similarly as in \cite[Section~6]{BDHH:22}, additionally observing that the reduced regularity in the $h$- and $a$-component does not affect the arguments thanks to the smooth dependence of the coefficients on $\eps$, $h$ and $a$, we deduce the existence of a constant $C_A > 0$ such that
\begin{equation}\label{eq:est of differences in the auton setting}
    \biggl\| \bigl(\bAH(u_1) - \bAH(u_2)\bigr)\vice - \bigl(B_1(u_1) - B_1(u_2)\bigr)\binom{h}{a} \biggr\|_{\rL^q(\Omega)} \le C_A \cdot \| u_1 - u_2 \|_{\rX_\gamma} \cdot \| w \|_{\rX_1}.
\end{equation}
Recalling from \autoref{rem:ref sol sea ice para-hyper}(c) that $u_0^*(t) \in V$ for all $t \in [0,T]$ and $\tu(t) = \hu(t) + u_0^*(t) \in V$ for all $t \in [0,T]$, we can use \eqref{eq:est of differences in the auton setting} in the case that $u_1 = u_0^*$, $u_2 = u_0$ and $w = u_0^*$ to derive
\begin{equation*}
    \begin{aligned}
        &\quad \biggl\| \bigl(\bAH(u_0^*) - \bAH(u_0)\bigr)\vii^* - \bigl(B_1(u_0^*) - B_1(u_0)\bigr)\binom{h_0^*}{a_0^*}\biggr\|_{\rL^p(0,T;\rL^q(\Omega))}\\
        &= \biggl(\int_0^T \biggl\| \bigl(\bAH(u_0^*(t)) - \bAH(u_0)\bigr)\vii^*(t) - \bigl(B_1(u_0^*(t)) - B_1(u_0)\bigr)\binom{h_0^*(t)}{a_0^*(t)}\biggr\|_{\rL^q(\Omega)}^p \srd t\biggr)^{\nicefrac{1}{p}}\\
        &\le C_A \biggl(\int_0^T \| u_0^*(t) - u_0 \|_{\rX_\gamma}^p \cdot \| u_0^*(t) \|_{\rX_1}^p \srd t\biggr)^{\nicefrac{1}{p}}\\
        &\le C \cdot \| u_0^* - u_0 \|_{\rBUC([0,T];\rX_\gamma)} \cdot \| u_0^* \|_{\bE_1}
        \le C \cdot \| u_0^* - u_0 \|_{\rBUC([0,T];\rX_\gamma)} C_T^*.
    \end{aligned}
\end{equation*}
This proves the first part of the assertion.
The second part can be shown likewise.
\end{proof}

The following proposition establishes an estimate of the term $F_1$.

\begin{prop}\label{prop:self map est F_1 sea ice para-hyper}
Let $p,q \in (1,\infty)$ satisfy \eqref{eq:cond p and q sea ice para-hyper} and $\tu \coloneqq \hu + u_0^*$, with $\hu \in \prescript{}{0}{\bE_1}$, $\| \hu \|_{\bE_1} \le R$ and~$u_0^*$ representing the reference solution from \autoref{prop:ref sol sea ice para-hyper}.
Moreover, suppose that $\Vatm$, $\Vocn$, $\nablaH H$ and~$\fgr$ fulfill \autoref{ass:external forcing terms}.
Recall the $T$-independent maximal regularity constant $\Cmr > 0$ from \autoref{cor:max reg sea ice para-hyper}.
Then there is $C_{F_1}(R,T) > 0$, with $C_{F_1}(R,T) < \nicefrac{R}{6 \Cmr}$ for $R > 0$ and $T > 0$ sufficiently small, such that~$\| F_1(\hu) \|_{\rL^p(0,T;\rL^q(\Omega))} \le C_{F_1}(R,T)$.
\end{prop}

\begin{proof}
Before starting with the estimates, let us first comment on some commonly used embeddings and notation.
We heavily rely on $\rW^{1,p}(0,T) \hookrightarrow \rL^\infty(0,T)$.
In addition, we frequently use $\| \hu \|_{\bE_1} \le R$ and often invoke the abbreviation $C_T^*$ of the norm of the reference solution.
By $C_0 \coloneqq \| u_0 \|_{\rX_\gamma}$, we denote the norm of the initial values in the sequel.
Instead of $Y_{\tvice}$, we write $Y$ throughout this proof, since we only consider $\tu = \hu + u_0^*$.
First, \autoref{lem:aux results sea ice para-hyper}(c) implies
\begin{equation*}
    \left\| \omega \tvice - \ccor \tvice^\perp \right\|_{\rL^p(0,T;\rL^q(\Omega))} \le C (T^{\nicefrac{1}{p}} \cdot \|  \hvice \|_{\rL^\infty(0,T;\rL^q(\Omega))} + \| \vii^* \|_{\rL^p(0,T;\rL^q(\Omega))}) \le C (T^{\nicefrac{1}{p}} R + C_T^*).
\end{equation*}
By \autoref{ass:external forcing terms}(b), we further get $\| -g \nablaH H \|_{\rL^p(0,T;\rL^q(\Omega))} \le T^{\nicefrac{1}{p}} \cdot \| -g \nablaH H \|_{\rL^\infty(0,T;\rL^q(\Omega))} \le C T^{\nicefrac{1}{p}}$.

For the subsequent considerations, it is useful to estimate the factor $\nicefrac{1}{\rice(\hh + h_0^*)}$ in $\rL^\infty(0,T;\rL^\infty(\Omega))$.
By virtue of \autoref{lem:final sol contained open set V sea ice para-hyper}, we especially obtain $\th(t) = \hh(t) + h_0^*(t) > \kappa$ for all $t \in [0,T]$, yielding
\begin{equation}\label{eq:est of frac with h sea ice para-hyper}
    \biggl\| \frac{1}{\rice \th} \biggr\|_{\rL^\infty(0,T;\rL^\infty(\Omega))} \le C.
\end{equation}
From \eqref{eq:est of frac with h sea ice para-hyper}, \autoref{lem:aux results sea ice para-hyper}(c) and \autoref{ass:external forcing terms}(a), we deduce that
\begin{equation*}
    \biggl\| \frac{1}{\rice \th} \tatm \biggr\|_{\rL^p(0,T;\rL^q(\Omega))} \le C T^{\nicefrac{1}{p}} \cdot \| \Vatm \|_{\rL^\infty(0,T;\rL^{2q}(\Omega))}^2 \le C T^{\nicefrac{1}{p}}.
\end{equation*}
Next, let us recall the well known estimate
\begin{equation}\label{eq:est BUC Xgamma by init data and max reg space sea ice para-hyper}
    \| u_0^* \|_{\rBUC([0,T];\rX_\gamma)} \le C(\| u_0 \|_{\rX_\gamma} + \| u_0^* \|_{\bE_1}) \le C(C_0 + C_T^*),
\end{equation}
where the constant $C > 0$ appearing in the estimate is independent of $T > 0$.
It then follows from \eqref{eq:est of frac with h sea ice para-hyper}, \autoref{lem:aux results sea ice para-hyper}(c), \autoref{lem:aux results sea ice para-hyper}(b), \eqref{eq:est BUC Xgamma by init data and max reg space sea ice para-hyper} and \autoref{ass:external forcing terms}(a) that
\begin{equation*}
    \begin{aligned}
        \biggl\| \frac{1}{\rice \th} \tocn(\tvice) \biggr\|_{\rL^p(0,T;\rL^q(\Omega))}
        &\le C T^{\nicefrac{1}{p}}\bigl(\| \Vocn \|_{\rL^\infty(0,T;\rL^{2q}(\Omega))}^2 + \| \hvice \|_{\rL^\infty(0,T;\rL^{2q}(\Omega))}^2 + \| \vii^* \|_{\rL^\infty(0,T;\rL^{2q}(\Omega))}^2\bigr)\\
        &\le C T^{\nicefrac{1}{p}}\left(1 + R^2 + C_0^2 + (C_T^*)^2\right).
    \end{aligned}
\end{equation*}

The term to be still estimated is $\bigl(\tbAH(\tu) - \bAH(u_0)\bigr)\tvice - \bigl(\tB(\tu) - B_1(u_0)\bigr)\binom{\th}{\ta}$.
This comes down to
\begin{equation}\label{eq:remaining term to est sea ice para-hyper}
    \bigl(\tbAH(\hu + u_0^*) - \bAH(\hu + u_0^*)\bigr)(\hvice + \vii^*) - \bigl(\tB(\hu + u_0^*) - B_1(\hu + u_0^*)\bigr)\binom{\hh + h_0^*}{\ha + a_0^*} \tand
\end{equation}
\begin{equation*}
    \bigl(\bAH(\hu + u_0^*) - \bAH(u_0^*)\bigr)(\hvice + \vii^*) - \bigl(B_1(\hu + u_0^*) - B_1(u_0^*)\bigr)\binom{\hh + h_0^*}{\ha + a_0^*} \taswellas
\end{equation*}
\begin{equation*}
    \bigl(\bAH(u_0^*) - \bAH(u_0)\bigr)(\hvice + \vii^*) - \bigl(B_1(u_0^*) - B_1(u_0)\bigr)\binom{\hh + h_0^*}{\ha + a_0^*}.
\end{equation*}
By \autoref{lem:ests of auton terms sea ice para-hyper}, the latter two terms admit estimates by $C R (R + C_T^*)$ and $C \| u_0^* - u_0 \|_{\rBUC([0,T];\rX_\gamma} (R + C_T^*)$, respectively, and we observe that $\| u_0^* - u_0 \|_{\rBUC([0,T];\rX_\gamma)}$ tends to zero as $T \to 0$ by \autoref{rem:ref sol sea ice para-hyper}(c).
It thus remains to estimate \eqref{eq:remaining term to est sea ice para-hyper} in the last part of the proof.

We start with $(\tbAH(\hu + u_0^*) - \bAH(\hu + u_0^*))(\hvice + \vii^*)$.
Below, we concentrate on the principal part and remark that the remaining part can be handled in a similar way.
Recalling the transformed Hibler operator $\tbAH$ from \eqref{eq:transformed Hibler op sea ice para-hyper} and the original Hibler operator from \eqref{eq:Hibler op diff form}, we note that we need to estimate
\begin{equation}\label{eq:est of principal part sea ice para-hyper}
    \begin{aligned}
        &\sum_{j,k,l,m=1}^2 \bigl(a_{ij}^{klm}(\teps(\hvice + \vii^*),P(\hh + h_0^*,\ha + a_0^*)) \del_m \teps_{jl}(\hvice + \vii^*)\\
        &\quad - a_{ij}^{kl}(\eps(\hvice + \vii^*),P(\hh + h_0^*,\ha + a_0^*)) \del_k \eps_{jl}(\hvice + \vii^*)\bigr) = \rI + \rII
    \end{aligned}
\end{equation}
in the sequel.
To ease notation, we write $\eps$ and $\teps$ as well as $P$, since the latter quantities always depend on $\hvice + \vii^*$ and $\hh + h_0^*$ as well as $\ha + a_0^*$, respectively.
The terms $\rI$ and $\rII$ in \eqref{eq:est of principal part sea ice para-hyper} are defined by
\begin{equation*}
    \rI \coloneqq \sum_{j,k,l,m=1}^2 \bigl(a_{ij}^{klm}(\teps,P) \del_m \teps_{jl} - a_{ij}^{kl}(\eps,P) \del_k \teps_{jl}\bigr) \tand \rII \coloneqq \sum_{j,k,l=1}^2\bigl(a_{ij}^{kl}(\eps,P) \del_k \teps_{jl} - a_{ij}^{kl}(\eps,P) \del_k \eps_{jl}\bigr).
\end{equation*}
With regard to $\rII$, it is necessary to estimate the difference of $\teps$ and $\eps$.
In fact, we compute the difference~$\del_k \teps_{jl} - \del_k \eps_{jl}$ to be given by 
\begin{equation*}
    \frac{1}{2}\sum_{n=1}^2\bigl((\del_j Y_n - \delta_{jn})\del_k \del_n \tv_{\ice,l} + (\del_l Y_n - \delta_{ln}) \del_k \del_n \tv_{\ice,j} + (\del_k \del_j Y_n) \del_n \tv_{\ice,j} + (\del_k \del_l Y_n) \del_n \tv_{\ice,j}\bigr).
\end{equation*}
From \autoref{lem:difference est of Id and nablaH Y}, we derive that
\begin{equation*}
    \begin{aligned}
        \| (\del_j Y_n - \delta_{jn})\del_k \del_n \tv_{\ice,l} \|_{\rL^p(0,T;\rL^q(\Omega))} 
        &\le C \cdot \| \Id_2 - \nablaH Y \|_{\rL^\infty(0,T;\rW^{1,q}(\Omega))} \bigl(\| \hu \|_{\bE_1} + \| u_0^* \|_{\bE_1}\bigr)\\
        &\le C T^{\nicefrac{1}{p'}} (R + C_T^*).
    \end{aligned}
\end{equation*}
The term $(\del_l Y_n - \delta_{ln}) \del_k \del_n \tv_{\ice,j}$ can be treated in the same way.
Employing the embedding of the trace space $\rX_\gamma$ from \eqref{eq:emd trace space sea ice para-hyper}, \autoref{lem:aux results sea ice para-hyper}(c), \autoref{lem:props of the trafo Euler Lagrange sea ice para-hyper}, the trace space embedding from \autoref{lem:aux results sea ice para-hyper}(a) with $T$-independent embedding constant and the estimate \eqref{eq:est BUC Xgamma by init data and max reg space sea ice para-hyper}, we argue that
\begin{equation*}
    \begin{aligned}
        \| (\del_k \del_j Y_n) \del_n \tv_{\ice,j} \|_{\rL^p(0,T;\rL^q(\Omega))}
        &\le \| (\del_k \del_j Y_n) \|_{\rL^p(0,T;\rL^q(\Omega))} \cdot \| \del_n \tv_{\ice,j} \|_{\rL^\infty(0,T;\rL^\infty(\Omega))}\\
        &\le C \cdot \| \nablaH Y \|_{\rL^p(0,T;\rW^{1,q}(\Omega))} \bigl(\| \hvice \|_{\rL^\infty(0,T;\rC^1(\oOmega))} + \| \vii^* \|_{\rL^\infty(0,T;\rC^1(\oOmega))}\bigr)\\
        &\le C T^{\nicefrac{1}{p}} \cdot \| \nablaH Y \|_{\rL^\infty(0,T;\rW^{1,q}(\Omega))}\bigl(\| \hu \|_{\rBUC(0,T;\rX_\gamma)} + \| u_0^* \|_{\rBUC(0,T;\rX_\gamma)}\bigr)\\
        &\le C T^{\nicefrac{1}{p}} (R + C_0 + C_T^*).
    \end{aligned}
\end{equation*}
As $(\del_k \del_l Y_n) \del_n \tv_{\ice,j}$ can be handled likewise, we have an estimate of the difference $\del_k \teps_{jl} - \del_k \eps_{jl}$.
The shape of the coefficients as introduced in \eqref{eq:coeffs principal part Hibler op}, $\tu(t) \in V$ for all $t \in [0,T]$ as stated in \autoref{lem:final sol contained open set V sea ice para-hyper} and the estimate of $\tu$ from \eqref{eq:bound of the norm tu sea ice para-hyper} reveal the boundedness, i.\ e., $\| a_{ij}^{kl}(\eps,P) \|_{\rL^\infty(0,T;\rL^\infty(\Omega))} \le C$.
A concatenation of the previous estimates yields the existence of $\beta > 0$ such that
\begin{equation*}
    \| \rII \|_{\rL^p(0,T;\rL^q(\Omega))} \le C T^\beta (R + C_0 + C_T^*).
\end{equation*}

Invoking $a_{ij}^{klm}$ from \eqref{eq:coeffs transformed Hibler sea ice para-hyper}, we further decompose $\rI$ into $\rIII$ and $\rIV$ given by
\begin{equation*}
    \rIII \coloneqq \sum_{j,k,l,m=1}^2 \bigl(a_{ij}^{klm}(\teps,P) - a_{ij}^{klm}(\eps,P)\bigr) \del_m \teps_{jl} \tand \rIV \coloneqq \sum_{j,k,l,m=1}^2 a_{ij}^{kl}(\del_k Y_m - \delta_{km}) \del_m \teps_{jl}.
\end{equation*}
With regard to $\rIV$, we conclude from \autoref{lem:difference est of Id and nablaH Y} that
\begin{equation}\label{eq:est del_k Y_m - delta_km sea ice para-hyper}
    \| \del_k Y_m - \delta_{km} \|_{\rL^\infty(0,T;\rL^\infty(\Omega))} \le C \cdot \| \Id_2 - \nablaH Y \|_{\rL^\infty(0,T;\rW^{1,q}(\Omega))} \le C T^{\nicefrac{1}{p'}}.
\end{equation}
Concerning $\del_m \teps_{jl}$, we can proceed in a similar way as above to derive $ \| \del_m \teps_{jl} \|_{\rL^p(0,T;\rL^q(\Omega))} \le C$.
Together with the above boundedness of the coefficients, this leads to
\begin{equation}\label{eq:est of rIV sea ice para-hyper}
    \| \rIV \|_{\rL^p(0,T;\rL^q(\Omega))} \le C T^{\nicefrac{1}{p'}}.
\end{equation}

The last term in $\rIII$ can be estimated as described above.
In addition, $\del_j Y_k$ is bounded from above by a constant in $\rL^\infty(0,T;\rL^\infty(\Omega))$ in view of \autoref{lem:props of the trafo Euler Lagrange sea ice para-hyper}.
Consequently, the task boils down to an estimate of the difference $a_{ij}^{kl}(\teps,P) - a_{ij}^{kl}(\eps,P)$.
The difference only concerns $\eps$ and $\teps$, so the first factor can be bounded similarly as above.
Thanks to the smooth dependence on $\eps$ in view of the regularization by~$\delta > 0$, we derive from the mean value theorem that
\begin{equation*}
    \biggl\| \frac{1}{\trid(\teps)}\left(\bS_{ij}^{kl} - \frac{1}{\trid^2(\teps)}(\bS \teps)_{ik} (\bS \teps)_{jl}\right) - \frac{1}{\trid(\eps)}\left(\bS_{ij}^{kl} - \frac{1}{\trid^2(\eps)}(\bS \eps)_{ik} (\bS \eps)_{jl}\right) \biggr\|_{\infty,\infty} \le C \cdot \| \teps - \eps \|_{\infty,\infty},
\end{equation*}
where $\| \cdot \|_{\infty,\infty} = \| \cdot \|_{\rL^\infty(0,T;\rL^\infty(\Omega))}$.
Moreover, we compute
\begin{equation*}
    \teps_{ij} - \eps_{ij} = \frac{1}{2} \sum_{k=1}^2 \bigl(((\del_i Y_k) - \delta_{ik}) \del_k \tv_{\ice,j} + ((\del_j Y_k) - \delta_{jk}) \del_k \tv_{\ice,i}\bigr).
\end{equation*}
The differences $(\del_i Y_k) - \delta_{ik}$ and $(\del_j Y_k) - \delta_{jk}$ can be handled with \eqref{eq:est del_k Y_m - delta_km sea ice para-hyper} to obtain an estimate by a $T$-power.
Concerning the derivative of $\tvice$, we employ \eqref{eq:emd trace space sea ice para-hyper}, \autoref{lem:aux results sea ice para-hyper}(a) and \eqref{eq:est BUC Xgamma by init data and max reg space sea ice para-hyper} to get
\begin{equation*}
    \| \del_k \tvice \|_{\rL^\infty(0,T;\rL^\infty(\Omega))} \le C \bigl(\| \hu \|_{\rBUC([0,T];\rX_\gamma)} + \| u_0^* \|_{\rBUC([0,T];\rX_\gamma)}\bigr) \le C(R + C_0 + C_T^*).
\end{equation*}
This yields that 
\begin{equation}\label{eq:est of rIII sea ice para-hyper}
    \| \rIII \|_{\rL^p(0,T;\rL^q(\Omega))} \le C T^{\nicefrac{1}{p'}}(R + C_0 + C_T^*).
\end{equation}
Putting together \eqref{eq:est of rIV sea ice para-hyper} and \eqref{eq:est of rIII sea ice para-hyper}, we conclude that
\begin{equation*}
    \| \rI \|_{\rL^p(0,T;\rL^q(\Omega))} \le C T^{\nicefrac{1}{p'}}(1 + R + C_0 + C_T^*).
\end{equation*}

Concerning the corresponding off-diagonal part $(\tB(\hu + u_0^*) - B_1(\hu + u_0^*)) \binom{\hh + h_0^*}{\ha + a_0^*}$, we calculate the $h$-part
\begin{equation*}
    \frac{\del_h P(\hh + h_0^*,\ha + a_0^*)}{2 \rice(\hh + h_0^*)} \sum_{j=1}^2 (\del_i Y_j - \delta_{ij}) \del_j (\hh + h_0^*).
\end{equation*}
Again, for the inverse of $2 \rice (\hh + h_0^*)$, we can use \eqref{eq:est of frac with h sea ice para-hyper}, while $\del_h P(\hh + h_0^*,\ha + a_0^*) = p^* \mre^{-c_\bullet(1-(\ha + a_0^*))}$.
Upon deducing from \autoref{lem:final sol contained open set V sea ice para-hyper} that $\ta(t) = \ha(t) + a_0^*(t) \in (0,1)$ for all $t \in [0,T]$, we find that
\begin{equation*}
    \| \del_h P(\hh + h_0^*,\ha + a_0^*) \|_{\rL^\infty(0,T;\rL^\infty(\Omega))} \le C.
\end{equation*}
In total, additionally exploiting the estimate of $\Id_2 - \nablaH Y$ from \autoref{lem:difference est of Id and nablaH Y} and \autoref{lem:aux results sea ice para-hyper}(c), we obtain
\begin{equation*}
    \begin{aligned}
        &\quad\biggl\| \frac{\del_h P(\hh + h_0^*,\ha + a_0^*)}{2 \rice(\hh + h_0^*)} \sum_{j=1}^2 (\del_i Y_j - \delta_{ij}) \del_j (\hh + h_0^*) \biggr\|_{\rL^p(0,T;\rL^q(\Omega))}\\
        &\le C \cdot \| \Id_2 - \nablaH Y \|_{\rL^\infty(0,T;\rL^\infty(\Omega))} \bigl(\| \hh \|_{\rL^p(0,T;\rW^{1,q}(\Omega))} + \| h_0^* \|_{\rL^p(0,T;\rW^{1,q}(\Omega))}\bigr)\\
        &\le C \cdot \| \Id_2 - \nablaH Y \|_{\rL^\infty(0,T;\rW^{1,q}(\Omega))}\bigl(T^{\nicefrac{1}{p}} \cdot \| \hh \|_{\rL^\infty(0,T;\rW^{1,q}(\Omega))} + \| u_0^* \|_{\bE_1}\bigr)
        \le C T^{\nicefrac{1}{p'}}\bigl(T^{\nicefrac{1}{p}} R + C_T^*\bigr).
    \end{aligned}
\end{equation*}
Analogously, we derive an estimate of the term $\frac{\del_a P(\hh + h_0^*,\ha + a_0^*)}{2 \rice(\hh + h_0^*)} \sum_{j=1}^2 (\del_i Y_j - \delta_{ij}) \del_j (\ha + a_0^*)$.

The assertion of the proposition then follows for some $C_{F_1}(R,T) > 0$.
A choice of $R > 0$ sufficiently small and then $T \to 0$ imply $C_{F_1}(R,T) < \nicefrac{R}{6 \Cmr}$ in view of \autoref{rem:ref sol sea ice para-hyper}(b) and~(c), where the latter ensure $C_T^* \to 0$ and $\| u_0^* - u_0 \|_{\rBUC([0,T];\rX_\gamma)} \to 0$.
\end{proof}

The next terms to be treated are $F_2$ and $F_3$.

\begin{prop}\label{prop:self map est F_2 and F_3 sea ice para-hyper}
Let $p,q \in (1,\infty)$ such that \eqref{eq:cond p and q sea ice para-hyper} is valid, consider $\tu \coloneqq \hu + u_0^*$, where $\hu \in \prescript{}{0}{\bE_1}$ and~$\| \hu \|_{\bE_1} \le R$, and $u_0^*$ represents the reference solution from \autoref{prop:ref sol sea ice para-hyper}.
Moreover, assume that~$\fgr$ fulfills \autoref{ass:external forcing terms}(c).
Then for the $T$-independent maximal regularity constant $\Cmr > 0$  from \autoref{cor:max reg sea ice para-hyper}, there are $C_{F_2}(R,T)$, $C_{F_3}(R,T) > 0$ with $C_{F_2}(R,T)$, $C_{F_3}(R,T) < \nicefrac{R}{6 \Cmr}$, for $R > 0$ and $T > 0$ sufficiently small, and $\| F_2(\hu) \|_{\rL^p(0,T;\rW^{1,q}(\Omega))} \le C_{F_2}(R,T)$ and $\| F_3(\hu) \|_{\rL^p(0,T;\rW^{1,q}(\Omega))} \le C_{F_3}(R,T)$.
\end{prop}

\begin{proof}
Conceptually, the proof parallels the one of \autoref{prop:self map est F_1 sea ice para-hyper}, and we often exploit the Banach algebra structure of $\rW^{1,q}(\Omega)$.
Thanks to $\omega > 0$ being fixed and \autoref{lem:aux results sea ice para-hyper}(c) and~(a), we first obtain
\begin{equation*}
    \| \omega \th \|_{\rL^p(0,T;\rW^{1,q}(\Omega))} \le C\bigl(T^{\nicefrac{1}{p}} \cdot \| \hu \|_{\rBUC([0,T];\rX_\gamma)} + C_T^*\bigr) \le C\bigl(T^{\nicefrac{1}{p}} R + C_T^*\bigr).
\end{equation*}
With regard to \autoref{ass:external forcing terms}(c) as well as \autoref{lem:aux results sea ice para-hyper}(c) and~(a), we then get
\begin{equation*}
    \begin{aligned}
        \| \Sh(\th,\ta) \|_{\rL^p(0,T;\rW^{1,q}(\Omega))}
        &\le \biggl\| \fgr\biggl(\frac{\th}{\ta}\biggr) \biggr\|_{\rL^\infty(0,T;\rW^{1,q}(\Omega))} \cdot \| \ha + a_0^* \|_{\rL^p(0,T;\rW^{1,q}(\Omega))}\\
        &\quad + \| \fgr(0) \|_{\rL^p(0,T;\rW^{1,q}(\Omega))} + \| \fgr(0)(\ha + a_0^*) \|_{\rL^p(0,T;\rW^{1,q}(\Omega))}\\
        &\le C\bigl(T^{\nicefrac{1}{p}} + \| \ha \|_{\rL^p(0,T;\rW^{1,q}(\Omega))} + \| a_0^* \|_{\rL^p(0,T;\rW^{1,q}(\Omega))}\bigr)\\
        &\le C\bigl(T^{\nicefrac{1}{p}}(1 + \| \hu \|_{\rBUC([0,T];\rX_\gamma)}) + C_T^*\bigr)
        \le C\bigl(T^{\nicefrac{1}{p}}(1 + R) + C_T^*\bigr).
    \end{aligned}
\end{equation*}

Hence, it remains to estimate the term $h_0 \divH \tvice - \th \sum_{j,k=1}^2 (\del_j Y_k) \del_k \tv_{\ice,j}$.
We treat this term by the addition and subtraction of the intermediate terms $(\hh + h_0^*) \divH(\hvice + \vii^*)$ as well as $h_0^* \divH(\hvice + \vii^*)$.
In that respect, we get
\begin{equation*}
    \begin{aligned}
        \| (h_0 - h_0^*) \divH (\hvice + \vii^*) \|_{\rL^p(0,T;\rW^{1,q}(\Omega))}
        &\le C \cdot \| h_0 - h_0^* \|_{\rBUC([0,T];\rW^{1,q}(\Omega))} \| \hvice + \vii^* \|_{\rL^p(0,T;\rW^{2,q}(\Omega))}\\
        &\le C \cdot \| u_0 - u_0^* \|_{\rBUC([0,T];\rX_\gamma)}(R + C_T^*).
    \end{aligned}
\end{equation*}
Concerning $h_0^* \divH(\hvice + \vii^*) - (\hh + h_0^*) \divH(\hvice + \vii^*) = \hh \divH(\hvice + \vii^*)$, \autoref{lem:aux results sea ice para-hyper}(a) yields
\begin{equation*}
    \begin{aligned}
        \| \hh \divH(\hvice + \vii^*) \|_{\rL^p(0,T;\rW^{1,q}(\Omega))}
        &\le C \cdot \| \hh \|_{\rL^\infty(0,T;\rW^{1,q}(\Omega))}(\| \hvice \|_{\rL^p(0,T;\rW^{2,q}(\Omega))} + \| \vii^* \|_{\rL^p(0,T;\rW^{2,q}(\Omega))})\\
        &\le C \cdot \| \hu \|_{\rBUC([0,T];\rX_\gamma)}\bigl(\| \hu \|_{\bE_1} + \| u_0^* \|_{\bE_1}\bigr)
        \le C R(R + C_T^*).
    \end{aligned}
\end{equation*}
The last term in the context of $F_2$ can be written as
\begin{equation*}
    \th \divH \tvice - \th \sum_{j,k=1}^2 (\del_j Y_k) \del_k (\hvice + \vii^*)_j = (\hh + h_0^*) \sum_{j,k=1}^2 (\delta_{jk} - \del_j Y_k) \del_k (\hvice + \vii^*)_j.
\end{equation*}
For its treatment, we make use of \autoref{lem:difference est of Id and nablaH Y}, the shape of $\rX_\gamma$ as indicated in \eqref{eq:trace space sea ice para-hyper}, \autoref{lem:aux results sea ice para-hyper}(a) and~\eqref{eq:est BUC Xgamma by init data and max reg space sea ice para-hyper} to derive
\begin{equation*}
    \begin{aligned}
        &\quad \left\| (\hh + h_0^*) \sum_{j,k=1}^2 (\delta_{jk} - \del_j Y_k) \del_k (\hvice + \vii^*)_j \right\|_{\rL^p(0,T;\rW^{1,q}(\Omega))}\\
        &\le C \cdot \| \hh + h_0^* \|_{\rL^\infty(0,T;\rW^{1,q}(\Omega))} \cdot \| \Id_2 - \nablaH Y \|_{\rL^\infty(0,T;\rW^{1,q}(\Omega))} \cdot \| \hvice + \vii^* \|_{\rL^p(0,T;\rW^{2,q}(\Omega))}\\
        &\le C T^{\nicefrac{1}{p'}} \bigl(\| \hu \|_{\rBUC([0,T];\rX_\gamma)} + \| u_0^* \|_{\rBUC([0,T];\rX_\gamma)}\bigr) \bigl(\| \hu \|_{\bE_1} + \| u_0^* \|_{\bE_1}\bigr)
        \le C T^{\nicefrac{1}{p'}}(R + C_0 + C_T^*) (R + C_T^*).
    \end{aligned}
\end{equation*}
The assertion of the proposition follows from a combination of the previous estimates and the observation that $C_T^* \to 0$ and $\| u_0^* - u_0 \|_{\rBUC([0,T];\rX_\gamma)} \to 0$ as $T \to 0$ by \autoref{rem:ref sol sea ice para-hyper}(b) and~(c).
\end{proof}

For completeness, we provide the differences of $F_1$, $F_2$ and $F_3$.
In fact, for $\hu_i = \tu_i - u_0^*$, $i=1,2$, we get
\begin{equation*}
    \begin{aligned}
        &F_1(\hu_1) - F_1(\hu_2)
        = \bigl(\tbAH(\tu_1) - \tbAH(\tu_2)\bigr) \tv_{\ice,1} - \bigl(\tB(\tu_1) - \tB(\tu_2)\bigr) \binom{\th_1}{\ta_1}\\
        &\quad + \bigl(\tbAH(\tu_2) - \bAH(u_0)\bigr)(\hv_{\ice,1} - \hv_{\ice,2}) -\bigl(\tB(\tu_2) - B_1(u_0)\bigr) \binom{\hh_1 - \hh_2}{\ha_1 - \ha_2} + \omega (\hv_{\ice,1} - \hv_{\ice,2})\\
        &\quad - \ccor (\hv_{\ice,1} - \hv_{\ice,2})^\perp + \biggl(\frac{1}{\rice \th_1} - \frac{1}{\rice \th_2}\biggr)\bigl(\tatm + \tocn(\tv_{\ice,1})\bigr) + \frac{1}{\rice \th_2}\bigl(\tocn(\tv_{\ice,1}) - \tocn(\tv_{\ice,2})\bigr),
    \end{aligned}
\end{equation*}
\begin{equation*}
    \begin{aligned}
        F_2(\hu_1) - F_2(\hu_2)
        &= -\sum_{j,k=1}^2 \bigl[\th_1(\del_j Y_k^{(1)}) - \th_2 (\del_j Y_k^{(2)})\bigr] \del_k (\tv_{\ice,1})_j -\biggl[\th_2 \sum_{j,k=1}^2 (\del_j Y_k^{(2)}) \del_k (\hv_{\ice,1} - \hv_{\ice,2})_j\\
        &\qquad - h_0 \divH(\hv_{\ice,1} - \hv_{\ice,2}) \biggr] + \omega (\hh_1 - \hh_2) + \Sh(\th_1,\ta_1) - \Sh(\th_2,\ta_2) \tand
    \end{aligned}
\end{equation*}
\begin{equation*}
    \begin{aligned}
        F_3(\hu_1) - F_3(\hu_2)
        &= -\sum_{j,k=1}^2 \bigl[\ta_1(\del_j Y_k^{(1)}) - \ta_2 (\del_j Y_k^{(2)})\bigr] \del_k (\tv_{\ice,1})_j -\biggl[\ta_2 \sum_{j,k=1}^2 (\del_j Y_k^{(2)}) \del_k (\hv_{\ice,1} - \hv_{\ice,2})_j\\
        &\qquad- a_0 \divH(\hv_{\ice,1} - \hv_{\ice,2}) \biggr] + \omega (\ha_1 - \ha_2) + \Sa(\th_1,\ta_1) - \Sa(\th_2,\ta_2).
    \end{aligned}
\end{equation*}

The lemma below discusses the Lipschitz continuity of $F_1$ from \eqref{eq:right-hand sides sea ice para-hyper}.
It can be obtained in a similar way as \autoref{prop:self map est F_1 sea ice para-hyper} when replacing the estimates from \autoref{lem:difference est of Id and nablaH Y}(a) by the (b)-part in order to estimate the difference in the diffeomorphism.
We also refer here to \cite[Lemma~6.4.8]{Bra:24} for details.

\begin{prop}\label{prop:Lipschitz est F_1 sea ice para-hyper}
Let $p,q \in (1,\infty)$ satisfy \eqref{eq:cond p and q sea ice para-hyper}, and let $\tu_i \coloneqq \hu_i + u_0^*$, $i = 1,2$, where $\hu_i \in \prescript{}{0}{\bE_1}$ with~$\| \hu_i \|_{\bE_1} \le R$, and $u_0^*$ denotes the reference solution from \autoref{prop:ref sol sea ice para-hyper}.
In addition, suppose that $\Vatm$, $\Vocn$, $\nablaH H$ and $\fgr$ fulfill \autoref{ass:external forcing terms}.
Then there exists $L_{F_1}(R,T) > 0$ with $L_{F_1}(R,T) \to 0$ as $R \to 0$ and $T \to 0$ as well as $\| F_1(\hu_1) - F_1(\hu_2) \|_{\rL^p(0,T;\rL^q(\Omega))} \le L_{F_1}(R,T) \cdot \| \hu_1 - \hu_2 \|_{\bE_1}$.
\end{prop}

Next, we state the Lipschitz estimates of the terms $F_2$ and $F_3$.
They can also be derived in a similar manner as \autoref{prop:self map est F_2 and F_3 sea ice para-hyper}, see also \cite[Lemma~6.4.9]{Bra:24}.

\begin{prop}\label{prop:Lipschitz est F_2 and F_3 sea ice para-hyper}
Let $p,q \in (1,\infty)$ fulfill \eqref{eq:cond p and q sea ice para-hyper}, and consider $\tu_i \coloneqq \hu_i + u_0^*$, $i = 1,2$, with $\hu_i \in \prescript{}{0}{\bE_1}$ and $\| \hu_i \|_{\bE_1} \le R$, and $u_0^*$ represents the reference solution from \autoref{prop:ref sol sea ice para-hyper}.
Moreover, assume that $\Vatm$, $\Vocn$, $\nablaH H$ and $\fgr$ satisfy \autoref{ass:external forcing terms}.
Then there are $L_{F_2}(R,T) > 0$ and $L_{F_3}(R,T) > 0$ with $L_{F_j}(R,T) \to 0$, $j=2,3$, as $R \to 0$ and $T \to 0$, $ \| F_2(\hu_1) - F_2(\hu_2) \|_{\rL^p(0,T;\rW^{1,q}(\Omega))} \le L_{F_2}(R,T) \cdot \| \hu_1 - \hu_2 \|_{\bE_1}$ as well as $\| F_3(\hu_1) - F_3(\hu_2) \|_{\rL^p(0,T;\rW^{1,q}(\Omega))} \le L_{F_3}(R,T) \cdot \| \hu_1 - \hu_2 \|_{\bE_1}$.
\end{prop}

\subsection{Proof of \autoref{thm:local strong wp Lagrangian coords} and \autoref{thm:local strong wp sea ice para-hyper}}\label{ssec:proof of main results}
\ 

This final subsection presents the proofs of the main theorem in Lagrangian and Eulerian coordinates.

\begin{proof}[Proof of \autoref{thm:local strong wp Lagrangian coords}]
We recall $\cK_T^R$ and $\Phi_T^R$ from \eqref{eq:set and map fixed point arg sea ice para-hyper}.
For $\hu \in \cK_T^R$, the maximal regularity from \autoref{cor:max reg sea ice para-hyper} and the self-map estimates from \autoref{prop:self map est F_1 sea ice para-hyper} and \autoref{prop:self map est F_2 and F_3 sea ice para-hyper} imply that
\begin{equation*}
    \| \Phi_T^R(\hu) \|_{\bE_1} \le \Cmr \cdot \| (F_1(\hu),F_2(\hu),F_3(\hu)) \|_{\bE_0} \le \Cmr \left(\frac{R}{6 \Cmr} + \frac{R}{6 \Cmr} + \frac{R}{6 \Cmr}\right) \le \frac{R}{2}
\end{equation*}
for $R > 0$ and $T > 0$ sufficiently small by virtue of the $T$-independence of $\Cmr > 0$.
Furthermore, from \autoref{prop:Lipschitz est F_1 sea ice para-hyper} and \autoref{prop:Lipschitz est F_2 and F_3 sea ice para-hyper}, we deduce the existence of a Lipschitz constant $L(R,T) > 0$ such that $L(R,T) \to 0$ as $R \to 0$ and $T \to 0$, and
\begin{equation*}
    \| \Phi_T^R(\hu_1) - \Phi_T^R(\hu_2) \|_{\bE_1} \le \Cmr L(R,T) \cdot \| \hu_1 - \hu_2 \|_{\bE_1} \le \frac{1}{2} \cdot \| \hu_1 - \hu_2 \|_{\bE_1}
\end{equation*}
for $\hu_1$, $\hu_2 \in \cK_T^R$ as well as $R > 0$ and $T > 0$ sufficiently small.
In total, $\Phi_T^R$ is a self-map and contraction on $\cK_T^R$, yielding the existence of a unique fixed point $u_* \in \prescript{}{0}{\bE_1}$ of \eqref{eq:linearized eq for fixed point sea ice para-hyper}.
In other words, \eqref{eq:transformed parabolic-hyperbolic regularized model} admits a unique solution $u_* \coloneqq \hu + u_0^*$.
Together with \autoref{lem:aux results sea ice para-hyper}(a), we obtain
\begin{equation*}
    \tu \in \bE_1 = \rW^{1,p}(0,T;\rX_0) \cap \rL^p(0,T;\rX_1) \hookrightarrow \rBUC([0,T];\rX_\gamma).
\end{equation*}
Finally, it is a consequence of \autoref{lem:final sol contained open set V sea ice para-hyper} that $\tu(t) \in V$ for all $t \in [0,T]$, completing the proof.
\end{proof}

The last step is to deduce \autoref{thm:local strong wp sea ice para-hyper} from \autoref{thm:local strong wp Lagrangian coords}.

\begin{proof}[Proof of \autoref{thm:local strong wp sea ice para-hyper}]
Because of $T \le T_1$, we argue that $X(t,\cdot)$ is in particular invertible on $\oOmega$ with inverse $Y(t,\cdot) = [X(t,\cdot)]^{-1}$.
From the representation of~$X$ in \eqref{eq:diffeo X Euler Lagrange sea ice para-hyper} as well as the regularity of $\tvice$ in \autoref{thm:local strong wp Lagrangian coords}, it follows that $X \in \rW^{2,p}(0,T;\rL^q(\Omega)^2) \cap \rW^{1,p}(0,T;\rW^{2,q}(\Omega)^2)$.
This is also valid for the inverse $Y$.
As a result, the variables in Eulerian coordinates $u = (\vice,h,a)$ can be recovered from those in Lagrangian coordinates by $\vice(t,\xH) \coloneqq \tvice(t,Y(t,\xH))$, $h(t,\xH) \coloneqq \th(t,Y(t,\xH))$ as well as $a(t,\yH) \coloneqq \ta(t,Y(t,\xH))$.
The proof is completed upon observing that the change of coordinates does not affect the regularity properties, and it has no influence on the property $u \in \rC([0,T];V)$ either, while the uniqueness is implied by the uniqueness of the fixed point joint with the uniqueness of the transform.
\end{proof}

We conclude this section by proving the continuous dependence on the initial data as asserted in \autoref{cor:cont dep init data} as well as the blow-up criterion from \autoref{cor:blow-up crit}.

\begin{proof}[Proof of \autoref{cor:cont dep init data}]
This result can be obtained by a slight modification of the fixed point argument.
Indeed, for $u_0^*$ denoting the reference solution associated to initial value $u_0$ and resulting from \autoref{prop:ref sol sea ice para-hyper} as well as $T > 0$ and $R > 0$, we set $\cK_{T,u_i}^R \coloneqq \left\{\overline{u} \in \bE_1 : \left. \overline{u}\right|_{t=0} = u_i \tand \| \overline{u} - u_0^* \|_{\bE_1} \le r\right\}$, $i=1,2$.
Moreover, we define $\Phi_{T,u_i}^R \colon \cK_{T,u_i}^R \to \bE_1$, $\Phi_{T,u_i}^R(\overline{u}) = \hu$ to be the solution to the linearized problem as in \eqref{eq:linearized eq for fixed point sea ice para-hyper}, with right-hand sides $F_1(\overline{u})$, $F_2(\overline{u})$ and $F_3(\overline{u})$ as introduced in \eqref{eq:right-hand sides sea ice para-hyper}, and with initial values $u_i = (v_{\ice,i},h_i,a_i)$.
Following the proof of \cite[Theorem~2.1]{KPW:10}, and adapting the arguments from the nonlinear estimates in \autoref{ssec:ests of nonlin terms}, for $R > 0$, $T > 0$ and $r > 0$ sufficiently small, we find that
\begin{equation*}
    \| \Phi_{T,u_1}^R (\overline{u}_1) - \Phi_{T,u_2}^R (\overline{u}_2) \|_{\bE_1} \le \frac{1}{2} \cdot \| \overline{u}_1 - \overline{u}_2 \|_{\bE_1} + C \cdot \| u_1 - u_2 \|_{\rX_\gamma}
\end{equation*}
for all $u_1$, $u_2 \in \obB_{\rX_\gamma}(u_0,r)$ as well as $\overline{u}_1 \in \cK_{T,u_1}^R$ and $\overline{u}_2 \in \cK_{T,u_2}^R$.
Applying the preceding estimate to the solutions $u(\cdot,u_1)$, $u(\cdot,u_2)$, whose existence can be obtained in a similar manner as in the proof of \autoref{thm:local strong wp Lagrangian coords} upon slightly modifying the arguments, we deduce the assertion of the corollary.
\end{proof}

The second result follows similarly as \cite[Corollary~5.1.2]{PS:16}.
For convenience, we provide a short proof.

\begin{proof}[Proof of \autoref{cor:blow-up crit}]
For $u_0 \in V$, define $t_+(u_0) \coloneqq \sup\{t > 0 : \eqref{eq:transformed parabolic-hyperbolic regularized model} \text{ has a solution on } [0,t]\}$, and assume that $t_+(u_0) < \infty$, $\dist_{\rX_\gamma}(\tu(t),\del V) \ge \eta$ for some $\eta > 0$, and the solution $\tu(t)$ converges to some~$u_1 \coloneqq \tu(t_+(u_0)) \in V$ as $t \to t_+(u_0)$.
The set $\tu([0,t_+(u_0)]) \subset V$ is compact in $V$.
Observing that the maximal $\rL^p$-regularity of $A(u')$ as established in \autoref{cor:max reg sea ice para-hyper} holds for all $u' \in V \subset \rX_\gamma$ for $p$ and~$q$ satisfying \eqref{eq:cond p and q sea ice para-hyper}, we deduce from \autoref{thm:local strong wp Lagrangian coords} and its proof the existence of a uniform $\beta > 0$ such that the problem \eqref{eq:transformed parabolic-hyperbolic regularized model} with initial data $u(0) = \tu(s)$, $s \in [0,t_+(u_0)]$ admits a unique solution $\tu'$ in $\bE_1$ on the time interval $(0,\beta)$.
For fixed $s_0 \in (t_+(u_0)-\beta,t_+(u_0))$, the solution $\tu'(\tau)$ coincides with $\tu(s_0+\tau)$.
In other words, the solution $\tu$ is extended beyond $t_+(u_0)$, contradicting the definition of the latter.
\end{proof}

\section{Concluding remarks}\label{sec:outlook and open problems}

In this final section, we discuss possible extensions of the results presented in the previous sections, sketch possible difficulties and mention remaining problems.

In order to fully exploit the parabolic regularization, it would be possible to introduce time weights of the shape $t^{1-\mu}$ for $\mu \in (\nicefrac{1}{p},1]$ and to consider $\rL_\mu^p(0,T;\rX_0)$ instead of $\rL^p(0,T;\rX_0)$, where $f \in \rL_\mu^p(0,T;\rX_0)$ if and only if $t^{1-\mu} f \in \rL^p(0,T;\rX_0)$.
The spaces $\rW_\mu^{1,p}(0,T;\rX_0)$ and $\rL_\mu^p(0,T;\rX_1)$ are defined analogously.
The underlying data space then is $\bE_{0,\mu} \coloneqq \rL_\mu^p(0,T;\rX_0)$, while the associated time trace space becomes $\rX_{\gamma,\mu} \coloneqq (\rX_0,\rX_1)_{\mu-\nicefrac{1}{p},p}$, and the corresponding maximal regularity space reads as
\begin{equation*}
    \bE_{1,\mu} \coloneqq \rW_\mu^{1,p}(0,T;\rX_0) \cap \rL_\mu^p(0,T;\rX_1).
\end{equation*}
An implementation of time weights would require an adjustment of the nonlinear estimates.
For the theory of critical spaces in the context of evolution equations and weighted spaces, we also refer to the work of Pr\"uss, Simonett and Wilke \cite{PSW:18}.

The bounded $\Hinfty$-calculus of the operator matrix $A(u_1)$ from \eqref{eq:op matrix sea ice para-hyper} as proved in \autoref{thm:Hinfty sea ice para-hyper} paves the way for interesting applications such as the theory of Banach scales due to Amann \cite[Chapter~5]{Ama:95} or the consideration of the problem in the stochastic setting via stochastic maximal regularity, see also \cite{vNVW:12a} for more details on this property.

The investigation of global-in-time results as e.\ g.\ obtained in \cite[Theorem~2.3]{BDHH:22}, see also \cite[Theorem~3.6.6]{Bra:24}, is left to future study for the present problem.
In fact, there are several difficulties here.
First, the linearized operator matrix $A(u_*)$ at a constant equilibrium solution of the shape $u_* = (0,h_*,a_*) \in V$ seems to be ``more singular'' than in the fully parabolic situation.
Indeed, zero is no longer a semi-simple eigenvalue, and the lack of invertibility cannot be overcome as for instance in \cite[Section~7.2]{Bra:24} by restricting to functions with spatial average zero.
Another difficulty arises in the transfer of global-in-time results from the Lagrangian to the Eulerian setting.
For this, it is necessary to have a global-in-time Lagrangian change of coordinates at hand.
In \cite{DHMT:20}, Danchin, Hieber, Mucha and Tolksdorf address this aspect in the context of free boundary value problems in the half space by using Da Prato-Grisvard theory, allowing for maximal $\rL^1$-regularity in the critical (homogeneous) Besov space. 
For a different approach to such problems, we also refer to the work of Ogawa and Shimizu \cite{OS:24}.

One can also ask whether the results established in \cite{BH:23}, \cite{BBH:24a}, \cite{BBH:24b} remain valid in the present situation of the parabolic-hyperbolic problem.
We also leave this to future study.

\medskip 

{\bf Acknowledgements.}
The author would like to thank Tim Binz, Karoline Disser and Matthias Hieber for fruitful discussions.
Moreover, he acknowledges the support by the German National Academy of Sciences Leopoldina through the Leopoldina Fellowship Program with grant number~LPDS 2024-07 and the support by DFG project FOR~5528. 
Most of the results in this paper are contained in modified form in Chapter~6 of the PhD thesis \cite{Bra:24} of the author.

\medskip

{\bf Data Availability Statement.}
The manuscript has no associated data.

\end{document}